\documentclass[12pt]{article}
\usepackage{amsmath}
\usepackage{amsfonts}
\usepackage{mitpress}

\newtheorem{theorem}{Theorem}

\newtheorem{lemma}[theorem]{Lemma}

\newtheorem{remark}[theorem]{Remark}

\newenvironment{proof}[1][Proof]{\noindent\textbf{#1.} }{\ \rule{0.5em}{0.5em}}
\newdimen\dummy
\dummy=\oddsidemargin
\input{tcilatex}
\begin{document}

\title{A weak comparison principle for reaction-diffusion systems}
\author{Jos\'{e} Valero \\
{\small Centro de Investigaci\'{o}n Operativa}\\
{\small Universidad Miguel Hern\'{a}ndez, } {\small Avda. Universidad, s/n.
- Elche - 03202 }\\
{\small e-mail: jvalero@umh.es}}
\maketitle

\begin{abstract}
In this paper we prove a weak comparison principle for a reaction-diffusion
system without uniqueness of solutions. We apply the abstract results to the
Lotka-Volterra system with diffusion, a generalized logistic equation and to
a model of fractional-order chemical autocatalysis with decay. Morever, in
the case of the Lotka-Volterra system a weak maximum principle is given, and
a suitable estimate in the space of essentially bounded functions $L^{\infty
}$ is proved for at least one solution of the problem.
\end{abstract}

\textbf{Keywords: }comparison of solutions, reaction-diffusion systems,
parabolic equations

\textbf{Mathematics Subject Classification (2010)}: 35B09, 35B50, 35B51,
35K55, 35K57

\bigskip

\section{Introduction}

Comparison results for parabolic equations and ordinary differential
equations are well known in the literature. One of the important
applications of such kind of results is the theory of monotone dynamical
systems, which leads to a more precise characterization of $\omega $-limit
sets and attractors. In the last years several authors have been working in
this direction (see, for example, \cite{ArrietaRodVidal}, \cite{KrauseRanft}%
, \cite{RobRodVid}, \cite{RodBerVidal}, \cite{Smith} for the deterministic
case, and \cite{ArChu}, \cite{Archu2}, \cite{Chues}, \cite{Kotelenetz} for
the stochastic case). In all these papers it is considered the classical
situation where the initial-value problem possesses a unique solution.

However, the situation is more complicated when we consider a differential
equation for which uniqueness of the Cauchy problem fails (or just it is not
known to hold). Let us consider an abstract parabolic problem%
\begin{equation}
\left\{ 
\begin{array}{c}
\dfrac{du}{dt}=A\left( t,u\left( t\right) \right) ,\text{ }\tau \leq t\leq T,
\\ 
u\left( \tau \right) =u_{\tau },%
\end{array}%
\right.  \label{Eq}
\end{equation}%
for which we can prove that for every initial data in the phase space $X$
(with a partial order $\leq )$ there exists at least one solution.

If we try to compare solutions of (\ref{Eq}) for two ordered initial data $%
u_{\tau }^{1}\leq u_{\tau }^{2}$, then we can consider a strong comparison
principle and a weak one.

The strong version would imply the existence of a solution $u_{1}$ with $%
u_{1}\left( \tau \right) =u_{\tau }^{1}$ such that 
\begin{equation}
u_{1}\left( t\right) \leq u_{2}\left( t\right) \text{ for }t\in \lbrack \tau
,T],  \label{Ineq}
\end{equation}%
for any solution $u_{2}$ with $u_{2}\left( \tau \right) =u_{\tau }^{2}$,
and, viceversa, the existence of a solution $u_{2}$ with $u_{2}\left( \tau
\right) =u_{\tau }^{2}$ such that (\ref{Ineq}) is satisfied for any solution 
$u_{1}$ with $u_{1}\left( \tau \right) =u_{\tau }^{1}$. This kind of result
is established in \cite{CLV2005} for a delayed ordinary differential
equations, defining then a multivalued order-preserving dynamical system.

The weak version of the comparison principle says that if $u_{\tau }^{1}\leq
u_{\tau }^{2}$, then there exist two solutions $u_{1},u_{2}$ of (\ref{Eq})
such that $u_{1}\left( \tau \right) =u_{\tau }^{1}$, $u_{2}\left( \tau
\right) =u_{\tau }^{2}$, and (\ref{Ineq}) holds.

There is in fact an intermediate version of the comparison principle, which
says that if we fix a solution $u_{1}$ of (\ref{Eq}) with $u_{1}\left( \tau
\right) =u_{\tau }^{1}$, then there exists a solution $u_{2}$ with $%
u_{2}\left( \tau \right) =u_{\tau }^{2}$ such that (\ref{Ineq}) is satisfied
(and viceversa). This is proved in \cite{CarvalhoGentile} for a differential
inclusion generated by a subdifferential map.

In this paper we establish a weak comparison principle for a
reaction-diffusion system in which the nonlinear term satisfies suitable
dissipative and growth conditions, ensuring existence of solutions but not
uniqueness. This principle is applied to several well known models in
Physics and Biology. Namely, a weak comparison of solutions is proved for
the Lotka-Volterra system, the generalized logistic equation and for a model
of fractional-order chemical autocatalysis with decay. Morever, in the case
of the Lotka-Volterra system a weak maximum principle is given, and a
suitable estimate in the space of essentially bounded functions $L^{\infty }$
is proved for at least one solution of the problem.

We note that in the papers \cite{KapVal06}, \cite{KapVal09} the existence of
a global attractor is proved for such kind of reaction-diffusion systems. In
a near future we will apply these results to obtain theorems concerning the
structure of the global attractor.

\section{Comparison results for reaction-diffusion systems}

We shall denote by $\left\vert \text{\textperiodcentered }\right\vert $ and $%
\left( \text{\textperiodcentered ,\textperiodcentered }\right) $ the norm
and scalar product in the space $\mathbb{R}^{m}$, $m\geq 1$. Let $d>0$ be an
integer and $\Omega \subset \mathbb{R}^{N}$ be a bounded open subset with
smooth boundary. Consider the problem 
\begin{equation}
\left\{ 
\begin{array}{l}
\dfrac{\partial u}{\partial t}-a\Delta u+f(t,u)=h(t,x),\ \ \ (t,x)\in (\tau
,T)\times \Omega , \\ 
u|_{x\in \partial \Omega }=0, \\ 
u|_{t=\tau }=u_{\tau }(x),%
\end{array}%
\right.  \label{Equation}
\end{equation}%
where $\tau ,T\in \mathbb{R}$, $T>\tau $, $x\in \Omega $, $u=\left(
u^{1}\left( t,x\right) ,...,u^{d}\left( t,x\right) \right) ,$ $f=\left(
f^{1},...,f^{d}\right) $, $a$ is a real $d\times d$ matrix with a positive
symmetric part $\dfrac{a+a^{t}}{2}\geq \beta I$, $\beta >0$, $h\in
L^{2}(\tau ,T;\left( L^{2}\left( \Omega \right) \right) ^{d})$. Moreover, $%
f=\left( f^{1}(t,u),...,f^{d}\left( t,u\right) \right) $ is jointly
continuous on $[\tau ,T]\times \mathbb{R}^{d}$ and satisfies the following
conditions: 
\begin{equation}
\sum_{i=1}^{d}|f^{i}(t,u)|^{\frac{p_{i}}{p_{i}-1}}\leq
C_{1}(1+\sum_{i=1}^{d}|u^{i}|^{p_{i}}),  \label{Cond1}
\end{equation}%
\begin{equation}
\left( f(t,u),u\right) \geq \alpha \sum_{i=1}^{d}|u^{i}|^{p_{i}}-C_{2},
\label{Cond2}
\end{equation}%
where $p_{i}\geq 2$, $\alpha ,C_{1},C_{2}>0$.

Let $H=\left( L^{2}\left( \Omega \right) \right) ^{d}$, $V=\left(
H_{0}^{1}\left( \Omega \right) \right) ^{d}$, and let $V^{\prime }$ be the
dual space of $V$. By $\left\Vert \text{\textperiodcentered }\right\Vert ,$ $%
\left\Vert \text{\textperiodcentered }\right\Vert _{V}$ we denote the norm
in $H$ and $V$, respectively. For $p=\left( p_{1},...,p_{d}\right) $ we
define the spaces 
\begin{align*}
L^{p}\left( \Omega \right) & =L^{p_{1}}\left( \Omega \right) \times \cdots
\times L^{p_{d}}\left( \Omega \right) , \\
L^{p}\left( \tau ,T;L^{p}\left( \Omega \right) \right) & =L^{p_{1}}\left(
\tau ,T;L^{p_{1}}\left( \Omega \right) \right) \times \cdots \times
L^{p_{d}}\left( \tau ,T;L^{p_{d}}\left( \Omega \right) \right) .
\end{align*}%
We take $q=\left( q_{1},...,q_{d}\right) $, where $\frac{1}{p_{i}}+\frac{1}{%
q_{i}}=1$.

We say that the function $u\left( \text{\textperiodcentered }\right) $ is a
weak solution of (\ref{Equation}) if $u\in L^{p}(\tau ,T;L^{p}(\Omega ))\cap 
$ $L^{2}(\tau ,T;V)\cap $ $C([\tau ,T];H),$ $\dfrac{du}{dt}\in L^{2}\left(
\tau ,T;V^{\prime }\right) +L^{q}\left( \tau ,T;L^{q}\left( \Omega \right)
\right) ,$ $u\left( \tau \right) =u_{\tau },$ and 
\begin{equation}
\int_{\tau }^{T}\left\langle \frac{du}{dt},\xi \right\rangle dt+\int_{\tau
}^{T}\int_{\Omega }\left( \nabla \left( au\right) ,\nabla \xi \right)
dxdt+\int_{\tau }^{T}\int_{\Omega }\left( f\left( t,u\right) ,\xi \right)
dxdt=\int_{\tau }^{T}\int_{\Omega }\left( h,\xi \right) dxdt,
\label{Solution}
\end{equation}%
for all $\xi \in L^{p}(\tau ,T;L^{p}(\Omega ))\cap $ $L^{2}(\tau ,T;V)$,
where $\left\langle \text{\textperiodcentered ,\textperiodcentered }%
\right\rangle $ denotes pairing in the space $V^{\prime }+L^{q}\left( \Omega
\right) $, and $\left( \nabla u,\nabla v\right) =\sum_{i=1}^{d}\left( \nabla
u^{i},\nabla v^{i}\right) .$

Under conditions (\ref{Cond1})-(\ref{Cond2}) it is known \cite[p.284]%
{ChepVishikBook} that for any $u_{\tau }\in H$ there exists at least one
weak solution $u=u(t,x)$ of (\ref{Equation}), and also that the function $%
t\mapsto \Vert u(t)\Vert ^{2}$ is absolutely continuous on $\left[ \tau ,T%
\right] $ and $\dfrac{d}{dt}\Vert u(t)\Vert ^{2}=2\left\langle \dfrac{du}{dt}%
,u\right\rangle $ for a.a. $t\in (\tau ,T)$.

Denote $r=\left( r_{1},...,r_{d}\right) ,$ $r_{i}=\max \left\{ 1;N\left(
1/q_{i}-1/2\right) \right\} $. Any weak solution satisfies $\dfrac{du}{dt}%
\in L^{q}\left( \tau ,T;H^{-r}\left( \Omega \right) \right) $ and 
\begin{equation*}
L^{q}(0,T;H^{-r}(\Omega ))=L^{q_{1}}(0,T;H^{-r_{1}}(\Omega ))\times \cdots
\times L^{q_{d}}(0,T;H^{-r_{d}}(\Omega )).
\end{equation*}

If, additionally, we assume that that $f\left( t,u\right) $ is continuously
differentiable with respect to $u$ for any $t\in \left[ \tau ,T\right] ,$ $%
u\in \mathbb{R}^{d},$ and 
\begin{equation}
\ \ \left( f_{u}(t,u)w,w\right) \geq -C_{3}(t)\left\vert w\right\vert ^{2}%
\text{, for all }w,u\in \mathbb{R}^{d},  \label{Cond3}
\end{equation}%
where $C_{3}(\cdot )\in L^{1}(\tau ,T)$, $C_{3}\left( t\right) \geq 0$, the
weak solution of (\ref{Equation}) is unique. Here, $f_{u}$ denotes the
jacobian matrix of $f.$

We consider also the following assumption: there exists $R_{0}>0$ such that%
\begin{equation}
f^{i}\left( t,u\right) \geq f^{i}\left( t,v\right) ,  \label{Cond4}
\end{equation}%
for any $t\in \lbrack \tau ,T]$ and any $u,v\in \mathbb{R}^{d}$ such that $%
u^{i}=v^{i}$ and $u^{j}\leq v^{j}$ if $j\neq i,$ and $\left\vert
u\right\vert ,\left\vert v\right\vert \leq R_{0}$, which means that the
systems is cooperative in the ball with radius $R_{0}$ centered at $0$.

Consider two problems%
\begin{equation}
\left\{ 
\begin{array}{l}
\dfrac{\partial u}{\partial t}-a\Delta u+f_{1}(t,u)=h_{1}(t,x),\ \ \
(t,x)\in (\tau ,T)\times \Omega , \\ 
u|_{x\in \partial \Omega }=0, \\ 
u|_{t=\tau }=u_{\tau }(x),%
\end{array}%
\right.  \label{P1}
\end{equation}%
\begin{equation}
\left\{ 
\begin{array}{l}
\dfrac{\partial u}{\partial t}-a\Delta u+f_{2}(t,u)=h_{2}(t,x),\ \ \
(t,x)\in (\tau ,T)\times \Omega , \\ 
u|_{x\in \partial \Omega }=0, \\ 
u|_{t=\tau }=u_{\tau }(x),%
\end{array}%
\right.  \label{P2}
\end{equation}%
where $f_{j}$ are jointly continuous on $[\tau ,T]\times \mathbb{R}^{d}.$
Among conditions (\ref{Cond1})-(\ref{Cond2}) and (\ref{Cond3})-(\ref{Cond4})
we shall consider the following: 
\begin{eqnarray}
h_{1}\left( t,x\right) &\leq &h_{2}\left( t,x\right) ,\text{ for a.a. }%
\left( t,x\right) \text{,}  \label{Ineqfh} \\
f_{1}^{i}\left( t,u\right) &\geq &f_{2}^{i}\left( t,u\right) ,\text{ for all 
}t,u.  \notag
\end{eqnarray}

\begin{lemma}
\label{ConstantsP}If $f_{j}$ satisfy (\ref{Cond1})-(\ref{Cond2}) and (\ref%
{Ineqfh}), then the constants $p_{i}$ has to be the same for $f_{1}$ and $%
f_{2}.$
\end{lemma}

\begin{proof}
Denote by $p_{i}^{j}$, $\alpha ^{j},\ C_{1}^{j}$, $C_{2}^{j}$ the constants
corresponding to $f_{j}$ in (\ref{Cond1})-(\ref{Cond2}). By contradiction
let, for example, $p_{1}^{2}>p_{1}^{1}$. Take the sequence $u_{n}=\left(
u_{n}^{1},0,...,0\right) ,$ where $u_{n}^{1}\rightarrow +\infty $ as $%
n\rightarrow \infty $. Then by (\ref{Cond1})-(\ref{Cond2}), (\ref{Ineqfh})
and Young's inequality we have%
\begin{equation*}
\alpha ^{2}\left\vert u_{n}^{1}\right\vert ^{p_{1}^{2}}-C_{2}^{2}\leq \left(
f_{2}(t,u_{n}),u_{n}\right) \leq \left( f_{1}(t,u_{n}),u_{n}\right)
\end{equation*}%
\begin{equation*}
=f_{1}^{1}\left( t,u_{n}\right) u_{n}^{1}\leq \left(
C_{1}^{1}(1+|u_{n}^{1}|^{p_{1}^{1}})\right) ^{\frac{p_{1}^{1}-1}{p_{1}^{1}}%
}u_{1}^{n}\leq C(1+|u_{n}^{1}|^{p_{1}^{1}}).
\end{equation*}%
But $p_{1}^{2}>p_{1}^{1}$ implies the existence of $n$ such that $\alpha
^{2}\left\vert u_{n}^{1}\right\vert
^{p_{1}^{2}}-C_{2}^{2}>C(1+|u_{n}^{1}|^{p_{1}^{1}})$, which is a
contradiction. Hence, $p_{1}^{2}\leq p_{1}^{1}.$

Conversely, let $p_{1}^{2}<p_{1}^{1}$. Then we take $u_{n}=\left(
u_{n}^{1},0,...,0\right) $ with $u_{n}^{1}\rightarrow -\infty $ as $%
n\rightarrow \infty $, so that%
\begin{equation*}
\alpha ^{1}\left\vert u_{n}^{1}\right\vert ^{p_{1}^{1}}-C_{2}^{1}\leq \left(
f_{1}(t,u_{n}),u_{n}\right) \leq \left( f_{2}(t,u_{n}),u_{n}\right)
\end{equation*}%
\begin{equation*}
=f_{2}^{1}\left( t,u_{n}\right) u_{n}^{1}\leq \left(
C_{1}^{2}(1+|u_{n}^{1}|^{p_{1}^{2}})\right) ^{\frac{p_{1}^{2}-1}{p_{1}^{2}}%
}u_{1}^{n}\leq C(1+|u_{n}^{1}|^{p_{1}^{2}}).
\end{equation*}%
As before, we obtain a contradiction, so $p_{1}^{2}=p_{1}^{1}.$

Repeating similar arguments for the other $p_{i}^{j}$ we obtain that $%
p_{i}^{1}=p_{i}^{2}$ for $i=1,...,d.$
\end{proof}

\bigskip

We recall \cite{KapVal06} that under conditions (\ref{Cond1})-(\ref{Cond2})
any solution $u\left( \text{\textperiodcentered }\right) $ of (\ref{P1})
satisfies the inequality%
\begin{equation}
\Vert u(t)\Vert ^{2}+2\beta \int\limits_{s}^{t}\Vert \nabla u(\tau )\Vert
^{2}d\tau +\alpha \sum_{i=1}^{d}\int\limits_{s}^{t}\Vert u^{i}(r)\Vert
_{L^{p_{i}}\left( \Omega \right) }^{p_{i}}dr\leq \Vert u(s)\Vert
^{2}+C\int\limits_{s}^{t}(\Vert h_{1}(r)\Vert ^{2}+1)dr,  \label{Est1}
\end{equation}%
for some constant $C>0.$ Of course, the same is valid for any solution of (%
\ref{P2}). From (\ref{Est1}), for any $T>\tau $ we obtain 
\begin{equation}
\left\Vert u\left( t\right) \right\Vert ^{2}\leq \Vert u_{\tau }\Vert
^{2}+C\int\limits_{\tau }^{T}(\Vert h_{1}(r)\Vert ^{2}+1)dr=K^{2}\left(
\Vert u_{\tau }\Vert ,\tau ,T\right) \text{ for all }\tau \leq t\leq T.
\label{Est2}
\end{equation}

We shall denote by $u_{1}(t)$ the solution of (\ref{P1}) corresponding to
the initial data $u_{\tau }^{1}$, and by $u_{2}(t)$ the solution of (\ref{P2}%
) corresponding to the initial data $u_{\tau }^{2}.$ Also, we take $%
u^{+}=\max \{u,0\}$ for $u\in \mathbb{R}.$

We obtain the following comparison result.

\begin{theorem}
\label{Comparison}Assume that $f_{j},h_{j}$ satisfy (\ref{Cond1})-(\ref%
{Cond2}), (\ref{Cond3}) and (\ref{Ineqfh}). If $u_{\tau }^{1}\leq u_{\tau
}^{2}$ and we suppose that $f_{2}$ satisfies (\ref{Cond4}) with $%
R_{0}^{2}\geq 2\max \{K^{2}\left( \Vert u_{\tau }^{1}\Vert ,\tau ,T\right)
,K^{2}\left( \Vert u_{\tau }^{2}\Vert ,\tau ,T\right) \}$, where $K\left(
\Vert u_{\tau }^{j}\Vert ,\tau ,T\right) $ is taken from (\ref{Est2}), we
have $u_{1}\left( t\right) \leq u_{2}\left( t\right) $, for all $t\in
\lbrack \tau ,T]$.
\end{theorem}

\begin{remark}
The results remains valid if, instead, $f_{1}$ satisfies (\ref{Cond4}) with $%
R_{0}^{2}\geq K^{2}\left( \Vert u_{\tau }^{1}\Vert ,T\right) .$
\end{remark}

\begin{remark}
If $f_{2}$ satisfies (\ref{Cond4}) for an arbitrary $R_{0}>0$ (that is, in
the whole space $\mathbb{R}^{d}$), then the result is true for any initial
data $u_{\tau }^{1}\leq u_{\tau }^{2}.$
\end{remark}

\begin{proof}
Let $g_{2}\left( t,u\right) =f_{2}\left( t,u\right) +C_{3}\left( t\right) u$%
. The function $g_{2}\left( t,\text{\textperiodcentered }\right) $ satisfies
(\ref{Cond3}) with $C_{3}\equiv 0$. For any $u_{1},u_{2}\in \mathbb{R}^{d}$
define $v_{2}\left( u_{1},u_{2}\right) $ by 
\begin{equation}
v_{2}^{i}=\left\{ 
\begin{array}{c}
u_{2}^{i}\text{, if }u_{1}^{i}\geq u_{2}^{i}, \\ 
u_{1}^{i}\text{, if }u_{1}^{i}<u_{2}^{i}.%
\end{array}%
\right.  \label{v2}
\end{equation}%
Note that $u_{1}-v_{2}=\left( u_{1}-u_{2}\right) ^{+}$ and $%
v_{2}-u_{2}=-\left( u_{2}-u_{1}\right) ^{+}$, so $v_{2}^{i}\leq u_{2}^{i}$
for all $i$. For the function $\left( u_{1}-u_{2}\right) ^{+}$ we can obtain
by (\ref{Ineqfh}) and the Mean Value Theorem that 
\begin{eqnarray*}
&&\frac{1}{2}\frac{d}{dt}\left\Vert \left( u_{1}-u_{2}\right)
^{+}\right\Vert ^{2}+a\left\Vert \left( u_{1}-u_{2}\right) ^{+}\right\Vert
_{V}^{2}\leq -\int_{\Omega }\left( f_{1}\left( t,u_{1}\right) -f_{2}\left(
t,u_{2}\right) ,\left( u_{1}-u_{2}\right) ^{+}\right) dx \\
&\leq &-\int_{\Omega }\left( f_{2}\left( t,u_{1}\right) -f_{2}\left(
t,v_{2}\right) ,\left( u_{1}-u_{2}\right) ^{+}\right) dx-\int_{\Omega
}\left( f_{2}\left( t,v_{2}\right) -f_{2}\left( t,u_{2}\right) ,\left(
u_{1}-u_{2}\right) ^{+}\right) dx \\
&=&-\int_{\Omega }\left( g_{2u}\left( t,v\left( t,x,u_{1},u_{2}\right)
\right) \left( u_{1}-v_{2}\right) ,\left( u_{1}-u_{2}\right) ^{+}\right)
dx+C_{3}\left( t\right) \left\Vert \left( u_{1}-u_{2}\right) ^{+}\right\Vert
^{2} \\
&&-\int_{\Omega }\left( f_{2}\left( t,v_{2}\right) -f_{2}\left(
t,u_{2}\right) ,\left( u_{1}-u_{2}\right) ^{+}\right) dx
\end{eqnarray*}%
\begin{eqnarray*}
&=&-\int_{\Omega }\left( g_{2u}\left( t,v\left( t,x,u_{1},u_{2}\right)
\right) \left( u_{1}-u_{2}\right) ^{+},\left( u_{1}-u_{2}\right) ^{+}\right)
dx+C_{3}\left( t\right) \left\Vert \left( u_{1}-u_{2}\right) ^{+}\right\Vert
^{2} \\
&&-\int_{\Omega }\left( f_{2}\left( t,v_{2}\right) -f_{2}\left(
t,u_{2}\right) ,\left( u_{1}-u_{2}\right) ^{+}\right) dx \\
&\leq &C_{3}\left( t\right) \left\Vert \left( u_{1}-u_{2}\right)
^{+}\right\Vert ^{2}-\int_{\Omega }\left( f_{2}\left( t,v_{2}\right)
-f_{2}\left( t,u_{2}\right) ,\left( u_{1}-u_{2}\right) ^{+}\right) dx.
\end{eqnarray*}

For all $t$ we have%
\begin{equation}
\left( f_{2}\left( t,v_{2}\right) -f_{2}\left( t,u_{2}\right) ,\left(
u_{1}-u_{2}\right) ^{+}\right) =\sum_{i\in J}\left( f_{2}^{i}\left(
t,v_{2}\right) -f_{2}^{i}\left( t,u_{2}\right) \right) \left(
u_{1}^{i}-u_{2}^{i}\right) ^{+},  \label{IneqCond4}
\end{equation}%
where $u_{1}^{i}-u_{2}^{i}>0$, for $i\in J$, and $u_{1}^{i}-u_{2}^{i}\leq 0$
if $i\not\in J$. For any $i\in J$ we have that $v_{2}^{i}=u_{2}^{i},$ and
then by $v_{2}\leq u_{2}$, $\left\vert v_{2}\right\vert ^{2}\leq \left\vert
u_{1}\right\vert ^{2}+\left\vert u_{2}\right\vert ^{2}$, (\ref{Est2}) and (%
\ref{Cond4}) we get%
\begin{equation*}
f_{2}^{i}\left( t,v_{2}\right) -f_{2}^{i}\left( t,u_{2}\right) \geq 0.
\end{equation*}

By Gronwall's lemma we get 
\begin{equation*}
\left\Vert \left( u_{1}\left( t\right) -u_{2}\left( t\right) \right)
^{+}\right\Vert ^{2}\leq \left\Vert \left( u_{\tau }^{1}-u_{\tau
}^{2}\right) ^{+}\right\Vert ^{2}e^{\int_{\tau }^{t}2C_{3}\left( s\right)
ds}=0.
\end{equation*}%
Thus $\left\Vert \left( u_{1}-u_{2}\right) ^{+}\right\Vert =0$, which means
that $u_{1}^{i}\left( x,t\right) -u_{2}^{i}\left( x,t\right) \leq 0$, for
a.a. $x\in \Omega $, and all $i\in \{1,..,d\},$ $t\in \lbrack \tau ,T].$
\end{proof}

\begin{remark}
In the scalar case, that is, $d=1,$ condition (\ref{Cond4}) is trivially
satisfied.
\end{remark}

\bigskip

When condition (\ref{Cond3}) fails to be true, we will obtain a weak
comparison principle.

Define a sequence of smooth functions $\psi _{k}:\mathbb{R}_{+}\rightarrow
\lbrack 0,1]$ satisfying 
\begin{equation}
\psi _{k}\left( s\right) =\left\{ 
\begin{array}{c}
1\text{, if }0\leq s\leq k, \\ 
0\leq \psi _{k}\left( s\right) \leq 1\text{, if }k\leq s\leq k+1, \\ 
0\text{, if }s\geq k+1.%
\end{array}%
\right.  \label{FiK}
\end{equation}

For every $k\geq 1$ we put $f_{k}^{i}(t,u)=\psi _{k}\left( \left\vert
u\right\vert \right) f^{i}\left( t,u\right) +\left( 1-\psi _{k}\left(
\left\vert u\right\vert \right) \right) g^{i}\left( u\right) $, where $%
g^{i}\left( u\right) =\left\vert u^{i}\right\vert ^{p_{i}-2}u^{i}$. Then $%
f_{k}\in \mathbb{C}([\tau ,T]\times \mathbb{R}^{d};\mathbb{R}^{d})$ and for
any $A>0,$ 
\begin{equation*}
\sup\limits_{t\in \lbrack \tau ,T]}\sup\limits_{|u|\leq
A}|f_{k}(t,u)-f(t,u)|\rightarrow 0,\ as\ k\rightarrow \infty .
\end{equation*}

Let $\rho _{\varepsilon }:\mathbb{R}^{d}\rightarrow \mathbb{R}_{+}$ be a
mollifier, that is, $\rho _{\epsilon }\in \mathbb{C}_{0}^{\infty }\left( 
\mathbb{R}^{d};\mathbb{R}\right) $, $supp\ \rho _{\epsilon }\subset
B_{\epsilon }=\{x\in \mathbb{R}^{d}:\left\vert x\right\vert <\epsilon \}$, $%
\int_{\mathbb{R}^{d}}\rho _{\epsilon }\left( s\right) ds=1$ and $\rho
_{\epsilon }\left( s\right) \geq 0$ for all $s\in \mathbb{R}^{d}$. We define
the functions 
\begin{equation*}
f_{k}^{\epsilon }(t,u)=\int_{\mathbb{R}^{d}}\rho _{\epsilon
}(s)f_{k}(t,u-s)ds.
\end{equation*}%
Since for any $k\geq 1$ $f_{k}$ is uniformly continuous on $[0,T]\times
\lbrack -k-1,k+1]$, there exist $\epsilon _{k}\in (0,1)$ such that for all $%
u $ satisfying $|u|\leq k$, and for all $s$ for which $|u-s|<\epsilon _{k}$
we have 
\begin{equation}
\sup\limits_{t\in \lbrack \tau ,T]}|f_{k}(t,u)-f_{k}(t,s)|\leq \frac{1}{k}.
\label{fkConverg}
\end{equation}%
We put $f^{k}(t,u)=f_{k}^{\epsilon _{k}}(t,u)$. Then $f^{k}(t,\cdot )\in 
\mathbb{C}^{\infty }(\mathbb{R}^{d};\mathbb{R}^{d})$, for all $t\in \lbrack
\tau ,T],k\geq 1$ .

For further arguments we need the following technical result \cite[Lemma 2]%
{KapVal09}.

\begin{lemma}
\label{Aux1}Let $f$ satisfy (\ref{Cond1})-(\ref{Cond2}). For all $k\geq 1$
the following statements hold: 
\begin{equation}
\ \sup\limits_{t\in \lbrack \tau ,T]}\sup\limits_{|u|\leq
A}|f^{k}(t,u)-f(t,u)|\rightarrow 0,\ \text{as }k\rightarrow \infty \text{, }%
\forall \ A>0,  \label{AuxEst1}
\end{equation}%
\begin{equation}
\sum_{i=1}^{d}|f^{ki}(t,u)|^{\frac{p_{i}}{p_{i}-1}}\leq
D_{1}(1+\sum_{i=1}^{d}|u^{i}|^{p_{i}}),\ \left( f^{k}(t,u),u\right) \geq
\gamma \sum_{i=1}^{d}|u^{i}|^{p_{i}}-D_{2},  \label{AuxEst2}
\end{equation}%
\begin{equation}
\left( f_{u}^{k}(t,u)w,w\right) \geq -D_{3}(k)\left\vert w\right\vert ^{2},%
\text{ }\forall u,w,  \label{AuxEst3}
\end{equation}%
where $D_{3}(k)$ is a non-negative number, and the positive constants $D_{1}$%
, $D_{2}\geq C_{2}$, $\gamma $ do not depend on $k$.
\end{lemma}

\bigskip

Consider first the scalar case.

\begin{theorem}
\label{ComparisonWeakScalar}Let $d=1$. Assume that $f_{j},h_{j}$ satisfy (%
\ref{Cond1})-(\ref{Cond2}) and (\ref{Ineqfh}). If $u_{\tau }^{1}\leq u_{\tau
}^{2},$ there exist two solutions $u_{1},u_{2}$ (of (\ref{P1}) and (\ref{P2}%
), respectively) such that $u_{1}\left( t\right) \leq u_{2}\left( t\right) $%
, for all $t\in \lbrack \tau ,T].$
\end{theorem}

\begin{proof}
For the functions $f_{j}$ we take the approximations $f_{j}^{k}$ (defined in
Lemma \ref{Aux1}), which satisfy (\ref{AuxEst1})-(\ref{AuxEst3}), and
consider the problems%
\begin{equation}
\left\{ 
\begin{array}{l}
\dfrac{\partial u}{\partial t}-a\Delta u+f_{j}^{k}(t,u)=h_{j}\left(
t,x\right) ,\ \ (t,x)\in (\tau ,T)\times \Omega , \\ 
u|_{x\in \partial \Omega }=0, \\ 
u|_{t=\tau }=u_{\tau },%
\end{array}%
\right.  \label{P1kb}
\end{equation}%
for $j=1,2.$ Problem (\ref{P1kb}) has a unique solution for any initial data 
$u_{\tau }\in H$. In view of Lemma \ref{ConstantsP} the constant $p$ is the
same for $f^{1}$ and $f^{2}$. We note that 
\begin{eqnarray*}
f_{1k}\left( t,u\right) &=&\psi _{k}\left( \left\vert u\right\vert \right)
f_{1}\left( t,u\right) +\left( 1-\psi _{k}\left( \left\vert u\right\vert
\right) \right) \left\vert u\right\vert ^{p-2}u \\
&\geq &\psi _{k}\left( \left\vert u\right\vert \right) f_{2}\left(
t,u\right) +\left( 1-\psi _{k}\left( \left\vert u\right\vert \right) \right)
\left\vert u\right\vert ^{p-2}u=f_{2k}\left( t,u\right) .
\end{eqnarray*}%
Then it is clear that $f_{1}^{k}(t,u)\geq f_{2}^{k}(t,u)$ for every $\left(
t,u\right) .$

By Theorem \ref{Comparison} we know that as $u_{\tau }^{1}\leq u_{\tau
}^{2}, $ we have $u_{1}^{k}\left( t\right) \leq u_{2}^{k}\left( t\right) $,
for all $t\in \lbrack \tau ,T]$, for the corresponding solutions of (\ref%
{P1kb}).

In view of Lemma \ref{Aux1} one can obtain in a standard way that (\ref{Est1}%
) is satisfied for the solutions of (\ref{P1kb}) with a constant $C$ not
depending on $k$ and replacing $\alpha $ by $\gamma $. Hence, the sequences $%
u_{j}^{k}\left( \text{\textperiodcentered }\right) $ are bounded in $%
L^{\infty }\left( \tau ,T;H\right) \cap L^{2}\left( \tau ,T;V\right) \cap
L^{p}\left( \tau ,T;L^{p}\left( \Omega \right) \right) $. It follows from (%
\ref{AuxEst2}) that $f_{jk}($\textperiodcentered $,u_{j}^{k}\left( \text{%
\textperiodcentered }\right) )$ are bounded in $L^{q}\left( \tau
,T;L^{q}\left( \Omega \right) \right) $ and also that $\{\dfrac{du^{k}}{dt}%
(\cdot )\}$ is bounded in $L^{q}(\tau ,T;H^{-r}(\Omega ))$, where $%
r_{i}=\max \{1;(\frac{1}{2}-\frac{1}{p_{i}})N\}$. By the Compactness Lemma 
\cite{Lions} we have that for some functions $u_{j}=u_{j}(t,x)$, $j=1,2$: 
\begin{equation}
u_{j}^{k}\rightarrow u_{j}\ \text{weakly star in}\ L^{\infty }(\tau ,T;H),
\label{Conv1}
\end{equation}%
\begin{equation}
u_{j}^{k}\rightarrow u_{j}\ \text{in}\ L^{2}(\tau ,T;H),\
u_{j}^{k}(t)\rightarrow u_{j}(t)\ \text{in}\ H\ \text{for\ a.a. }t\in (\tau
,T),  \label{Conv2}
\end{equation}%
\begin{equation}
u_{j}^{k}(t,x)\rightarrow u_{j}(t,x)\ \text{for a.a}.\ (t,x)\in (\tau
,T)\times \Omega ,  \label{Conv3}
\end{equation}%
\begin{equation}
u_{j}^{k}\rightarrow u_{j}\ \text{weakly in}\ L^{2}(\tau ,T;V),
\label{Conv4}
\end{equation}%
\begin{equation}
\dfrac{du_{j}^{k}}{dt}\rightarrow \dfrac{du_{j}}{dt}\text{ weakly in }%
L^{q}(\tau ,T;H^{-r}(\Omega )),  \label{Conv5}
\end{equation}%
\begin{equation}
u_{j}^{k}\rightarrow u_{j}\ \text{weakly in}\ L^{p}(\tau ,T;L^{p}\left(
\Omega \right) ).  \label{Conv5b}
\end{equation}%
Also, arguing as in \cite[p.3037]{MorVal2} we obtain%
\begin{equation}
u_{j}^{k}(t)\rightarrow u_{j}(t)\ \text{weakly in }H\ \text{for all }t\in
\lbrack \tau ,T].  \label{Conv6}
\end{equation}%
Moreover, by (\ref{AuxEst1}) and (\ref{Conv3}) we have $f_{jk}(t,u^{k}(t,x))%
\rightarrow f_{j}(t,u(t,x))$ for a.a. $(t,x)\in (\tau ,T)\times \Omega $ and
then the boundedness of $f_{jk}\left( \text{\textperiodcentered }%
,u_{j}^{k}\left( \text{\textperiodcentered }\right) \right) $ in $L^{q}(\tau
,T;L^{q}(\Omega ))$ implies that $f_{jk}($\textperiodcentered $%
,u_{j}^{k}\left( \text{\textperiodcentered }\right) )$ converges to $f($%
\textperiodcentered $,u($\textperiodcentered $))$ weakly in $L^{q}(\tau
,T;L^{q}(\Omega ))$ \cite{Lions}. It follows that $u_{1}(\cdot ),$ $%
u_{2}(\cdot )$ are weak solutions of (\ref{P1}) and (\ref{P2}),
respectively, with $u_{1}(\tau )=u_{\tau }^{1}$, $u_{2}\left( \tau \right)
=u_{\tau }^{2}$.

Moreover, one can prove that 
\begin{equation}
u_{j}^{k}(t)\rightarrow u_{j}(t)\ \text{strongly in }H\ \text{for all }t\in
\lbrack \tau ,T].  \label{Conv7}
\end{equation}%
Indeed, we define the functions $J_{jk}(t)=\Vert u_{j}^{k}(t)\Vert
^{2}-C\int\limits_{\tau }^{t}(\Vert h_{j}(s)\Vert ^{2}+1)ds$, $%
J_{j}(t)=\Vert u_{j}(t)\Vert ^{2}-C\int\limits_{\tau }^{t}(\Vert
h_{j}(s)\Vert ^{2}+1)ds$, which are non-increasing in view of (\ref{Est1}).
Also, from (\ref{Conv2}) we have $J_{jk}(t)\rightarrow J_{j}(t)$ for a.a. $%
t\in \left( \tau ,T\right) $. Then one can prove that $\lim
\sup_{k\rightarrow \infty }J_{jk}(t)\leq J_{j}(t)$ for all $t\in \lbrack
\tau ,T]$ (see \cite[p.623]{KapVal06} for the details). Hence, $\lim
\sup_{k\rightarrow \infty }\Vert u_{j}^{k}(t)\Vert \leq \Vert u_{j}(t)\Vert $%
. Together with (\ref{Conv6}) this implies (\ref{Conv7}) (see again \cite[%
p.623]{KapVal06} for more details).

Hence, passing to the limit we obtain 
\begin{equation*}
u_{1}\left( t\right) \leq u_{2}\left( t\right) ,\text{ for all }t\in \lbrack
\tau ,T].
\end{equation*}
\end{proof}

\bigskip

Further, let us prove the general case for an arbitrary $d\in \mathbb{N}.$

\begin{theorem}
\label{ComparisonWeak}Assume that $f_{j},h_{j}$ satisfy (\ref{Cond1})-(\ref%
{Cond2}) and (\ref{Ineqfh}). Also, suppose that either $f_{1}$ or $f_{2}$
satisfies (\ref{Cond4}) for an arbitrary $R_{0}>0$. If $u_{\tau }^{1}\leq
u_{\tau }^{2},$ there exist two solutions $u_{1},u_{2}$ (of (\ref{P1}) and (%
\ref{P2}), respectively) such that $u_{1}\left( t\right) \leq u_{2}\left(
t\right) $, for all $t\in \lbrack \tau ,T].$
\end{theorem}

\begin{proof}
Let $f_{1}$ be the function which satisfies (\ref{Cond4}). We take the
approximations $f_{1}^{k},\ f_{2}^{k}$ (defined in Lemma \ref{Aux1}), which
satisfy (\ref{AuxEst1})-(\ref{AuxEst3}). Then we consider problems (\ref%
{P1kb}).

In view of Lemma \ref{ConstantsP} the constants $p_{i}$ are the same for $%
f^{1}$ and $f^{2}$. We note that 
\begin{eqnarray*}
f_{1k}^{i}\left( t,u\right) &=&\psi _{k}\left( \left\vert u\right\vert
\right) f_{1}^{i}\left( t,u\right) +\left( 1-\psi _{k}\left( \left\vert
u\right\vert \right) \right) \left\vert u^{i}\right\vert ^{p_{i}-2}u^{i} \\
&\geq &\psi _{k}\left( \left\vert u\right\vert \right) f_{2}^{i}\left(
t,u\right) +\left( 1-\psi _{k}\left( \left\vert u\right\vert \right) \right)
\left\vert u^{i}\right\vert ^{p_{i}-2}u^{i}=f_{2k}^{i}\left( t,u\right) .
\end{eqnarray*}%
Then it is clear that $f_{1}^{k}(t,u)\geq f_{2}^{k}(t,u)$ for every $\left(
t,u\right) .$

Using Lemma \ref{Aux1} it is standard to obtain estimate (\ref{Est2}) with a
constant $C$ not depending on $k$. Hence, the solutions $u_{j}^{k}\left( 
\text{\textperiodcentered }\right) $ of (\ref{P1kb}) satisfy%
\begin{equation*}
\left\Vert u_{j}^{k}\left( t\right) \right\Vert ^{2}\leq \Vert u_{\tau
}^{j}\Vert ^{2}+C\int\limits_{\tau }^{T}(\Vert h_{j}(r)\Vert
^{2}+1)dr=K^{2}\left( \Vert u_{\tau }^{j}\Vert ,\tau ,T\right) .
\end{equation*}%
We note that 
\begin{equation*}
f_{1k}\left( t,u\right) =f_{1}\left( t,u\right) ,
\end{equation*}%
if $\left\vert u\right\vert \leq k$, since in such a case $\psi _{k}\left(
\left\vert u\right\vert \right) =1$. Hence, if $k^{2}\geq 2\max
\{K^{2}\left( \Vert u_{\tau }^{1}\Vert ,\tau ,T\right) ,K^{2}\left( \Vert
u_{\tau }^{2}\Vert ,\tau ,T\right) \}$ the functions $f_{1k}$ satisfy
condition (\ref{Cond4}) with $R_{0}=k.$ Therefore, for any $t\in \lbrack
\tau ,T]$ and any $u,v\in \mathbb{R}^{d}$ such that $u^{i}=v^{i}$ and $%
u^{j}\leq v^{j}$ if $j\neq i,$ and $\left\vert u\right\vert ,\left\vert
v\right\vert \leq k-1$, we have%
\begin{equation*}
f_{1}^{k}\left( t,u\right) =\int_{\mathbb{R}^{d}}\rho _{\epsilon
_{k}}(s)f_{1k}(t,u-s)ds\geq \int_{\mathbb{R}^{d}}\rho _{\epsilon
_{k}}(s)f_{1k}(t,v-s)ds=f_{2}^{k}\left( t,u\right) .
\end{equation*}%
Thus, if $\left( k-1\right) ^{2}\geq 2\max \{K^{2}\left( \Vert u_{\tau
}^{1}\Vert ,T\right) ,K^{2}\left( \Vert u_{\tau }^{2}\Vert ,T\right) \},$
the functions $f_{1}^{k}$ satisfy condition (\ref{Cond4}) with $R_{0}=k-1$.

By Theorem \ref{Comparison} we know that as $u_{\tau }^{1}\leq u_{\tau
}^{2}, $ we have $u_{1}^{k}\left( t\right) \leq u_{2}^{k}\left( t\right) $,
for all $t\in \lbrack \tau ,T]$, $k\geq 1+(2\max \{K^{2}\left( \Vert u_{\tau
}^{1}\Vert ,T\right) ,K^{2}\left( \Vert u_{\tau }^{2}\Vert ,T\right) \})^{%
\frac{1}{2}}$, for the corresponding solutions of (\ref{P1kb}).

Repeating the same proof of Theorem \ref{ComparisonWeakScalar} we obtain
that the sequences $u_{1}^{k},u_{2}^{k}$ converge (up to a subsequence) in
the sense of (\ref{Conv1})-(\ref{Conv7}) to the solutions $u_{1},u_{2}$ of
problem (\ref{P1}) and (\ref{P2}), respectively. Also, it holds%
\begin{equation*}
u_{1}\left( t\right) \leq u_{2}\left( t\right) ,\text{ for all }t\in \lbrack
\tau ,T].
\end{equation*}
\end{proof}

\bigskip

In the applications we need to generalize this theorem to the case where the
constant $\alpha $ can be negative. We shall do this when $f_{1},f_{2}$ have
sublinear growth (that is, $p_{i}=2$ for all $i$). Consider for (\ref%
{Equation}) the following conditions:%
\begin{equation}
\left\vert f\left( t,u\right) \right\vert \leq C_{1}\left( 1+\left\vert
u\right\vert \right) ,  \label{Cond1b}
\end{equation}%
\begin{equation}
\left( f\left( t,u\right) ,u\right) \geq \alpha \left\vert u\right\vert
^{2}-C_{2},  \label{Cond2b}
\end{equation}%
where $\alpha \in \mathbb{R}$, and $C_{1},C_{2}>0.$

Let $f_{1},f_{2}$ satisfy (\ref{Cond1b})-(\ref{Cond2b}) with constants $%
\alpha ^{j},C_{1}^{j},C_{2}^{j},\ j=1,2$. Then if $\min \{\alpha ^{1},\alpha
^{2}\}\leq 0$, we make in (\ref{Equation}) the change of variable $%
v=e^{-\beta t}u$, where $\beta >-\min \{\alpha ^{1},\alpha ^{2}\}$. Hence,
multiplying (\ref{P1}) and (\ref{P2}) by $e^{-\beta t}$ we have%
\begin{equation}
\left\{ 
\begin{array}{l}
\dfrac{\partial v}{\partial t}-a\Delta v+e^{-\beta t}f_{1}(t,e^{\beta
t}v)+\beta v=e^{-\beta t}h_{1}(t,x),\ \ \ (t,x)\in (\tau ,T)\times \Omega ,
\\ 
v|_{x\in \partial \Omega }=0, \\ 
v|_{t=\tau }=e^{-\beta \tau }u_{\tau }(x),%
\end{array}%
\right.  \label{P1b}
\end{equation}%
\begin{equation}
\left\{ 
\begin{array}{l}
\dfrac{\partial v}{\partial t}-a\Delta v+e^{-\beta t}f_{2}(t,e^{\beta
t}v)+\beta v=e^{-\beta t}h_{2}(t,x),\ \ \ (t,x)\in (\tau ,T)\times \Omega ,
\\ 
v|_{x\in \partial \Omega }=0, \\ 
v|_{t=\tau }=e^{-\beta \tau }u_{\tau }(x).%
\end{array}%
\right.  \label{P2b}
\end{equation}

It is easy to check that if $v\left( t\right) $ is a weak solution of (\ref%
{P1b}), then $u\left( t\right) =e^{\beta t}v\left( t\right) $ is a weak
solution of (\ref{P1}) (and the same is true, of course, for (\ref{P2b}) and
(\ref{P2})). Conversely, if $u\left( t\right) $ is a weak solution of (\ref%
{P1}), then $v\left( t\right) =e^{-\beta t}u\left( t\right) $ is a weak
solution of (\ref{P1b}).

The functions $\widetilde{f}_{j}\left( t,v\right) =e^{-\beta
t}f_{j}(t,e^{\beta t}v)+\beta v$ satisfy (\ref{Cond1})-(\ref{Cond2}) with $%
p_{i}=2$ for all $i$. Indeed,%
\begin{equation}
|\widetilde{f}_{j}(t,v)|\leq \leq e^{-\beta t}C_{1}^{j}\left( 1+e^{\beta
t}\left\vert v\right\vert \right) +\beta \left\vert v\right\vert \leq 
\widetilde{C}_{1}^{j}\left( 1+\left\vert v\right\vert \right) ,  \label{C1}
\end{equation}%
\begin{eqnarray}
\left( \widetilde{f}_{j}(t,v),v\right) &\geq &e^{-2\beta t}(f_{j}(t,e^{\beta
t}v),e^{\beta t}v)+\beta \left\vert v\right\vert ^{2}  \label{C2} \\
&\geq &\left( \alpha ^{j}+\beta \right) \left\vert v\right\vert
^{2}-C_{2}^{j},  \notag
\end{eqnarray}%
where $\alpha ^{j}+\beta >0$.

Then we obtain the following.

\begin{theorem}
\label{ComparisonWeak2}Assume that $f_{j},h_{j}$ satisfy (\ref{Cond1b})-(\ref%
{Cond2b}) and (\ref{Ineqfh}). Also, suppose that either $f_{1}$ or $f_{2}$
satisfies (\ref{Cond4}) for an arbitrary $R_{0}>0$. If $u_{\tau }^{1}\leq
u_{\tau }^{2},$ there exist two solutions $u_{1},u_{2}$ (of (\ref{P1}) and (%
\ref{P2}), respectively) such that $u_{1}\left( t\right) \leq u_{2}\left(
t\right) $, for all $t\in \lbrack \tau ,T].$
\end{theorem}

\begin{proof}
We consider problems (\ref{P1b}) and (\ref{P2b}). In view of (\ref{C1})-(\ref%
{C2}) $\widetilde{f}_{j}\left( t,v\right) =e^{-\beta t}f_{j}(t,e^{\beta
t}v)+\beta v$ satisfy (\ref{Cond1})-(\ref{Cond2}). Also, defining $%
\widetilde{h}_{j}\left( t,x\right) =e^{-\beta t}h_{j}(t,x)$ it is clear that
(\ref{Ineqfh}) holds. Finally, if, for example, $f_{1}$ satisfies (\ref%
{Cond4}) for any $R_{0}>0$, then it is obvious that for $\widetilde{f}_{1}$
is true as well.

Hence, by Theorem \ref{ComparisonWeak} there exist two solutions $%
v_{1},v_{2} $ (of (\ref{P1b}) and (\ref{P2b}), respectively), with $%
v_{j}\left( \tau \right) =e^{-\beta \tau }u_{\tau }^{j}$, such that $%
v_{1}\left( t\right) \leq v_{2}\left( t\right) $, for all $t\in \lbrack \tau
,T].$ Thus%
\begin{equation*}
u_{1}\left( t\right) =e^{\beta t}v_{1}\left( t\right) \leq e^{\beta
t}v_{2}\left( t\right) =u_{2}\left( t\right) \text{, for }t\in \lbrack \tau
,T],
\end{equation*}%
and $u_{1},u_{2}$ are solutions (of (\ref{P1}) and (\ref{P2}), respectively,
such that $u_{j}\left( \tau \right) =u_{\tau }^{j}$.
\end{proof}

\begin{remark}
If $f_{j}$ satisfy (\ref{Cond3}), then the solutions $u_{1},u_{2}$ given in
Theorem \ref{ComparisonWeak2} are unique for the corresponding initial data.
\end{remark}

\section{Comparison for positive solutions}

Denote $\mathbb{R}_{+}^{d}=\left\{ u\in \mathbb{R}^{d}:u^{i}\geq 0\right\} $%
. Let us consider the previous results in the case where the solutions have
to be positive. Consider now the following conditions:%
\begin{equation}
\text{The matrix }a\text{ is diagonal,}  \label{Cond10}
\end{equation}%
\begin{equation}
h^{i}\left( t,x\right) -f^{i}\left(
t,u^{1},...,u^{i-1},0,u^{i+1},...,u^{d}\right) \geq 0,  \label{Cond11}
\end{equation}%
for all $i,$ a.e. $\left( t,x\right) \in \left( \tau ,T\right) \times \Omega 
$ and $u^{j}\geq 0$ if $j\neq i.$ Obviously, in the scalar case these
conditions just mean that 
\begin{equation}
h\left( t,x\right) -f\left( t,0\right) \geq 0,  \label{Cond12}
\end{equation}%
for a.e. $\left( t,x\right) \in \left( \tau ,T\right) \times \Omega .$

It is well known (see \cite[Lemma 5]{KapVal09} for a detailed proof) that if
we assume conditions (\ref{Cond1})-(\ref{Cond2}) only\ for $u\in \mathbb{R}%
_{+}^{d}$, and also (\ref{Cond3}) and (\ref{Cond10})-(\ref{Cond11}), then
for any $u_{\tau }\geq 0$ there exists a unique weak solution $u\left( \text{%
\textperiodcentered }\right) $ of (\ref{Equation}). Moreover, $u\left( \text{%
\textperiodcentered }\right) $ is such that $u\left( t\right) \geq 0$ for
all $t\in \lbrack \tau ,T]$.

On the other hand, if we assume these conditions except (\ref{Cond3}), then
there exists at least one weak solution $u\left( \text{\textperiodcentered }%
\right) $ of (\ref{Equation}) such that $u\left( t\right) \geq 0$ for all $%
t\in \lbrack \tau ,T]$ \cite[Theorem 4]{KapVal09}. Moreover, we can prove
the following.

\begin{lemma}
\label{Unique}Assume conditions (\ref{Cond1})-(\ref{Cond2}), (\ref{Cond3})
only\ for $u\in \mathbb{R}_{+}^{d}$, and also (\ref{Cond10})-(\ref{Cond11}).
Then there exists a weak solution $u\left( \text{\textperiodcentered }%
\right) $ of (\ref{Equation}), which\ is unique in the class of solutions
satisfying $u\left( t\right) \geq 0$ for all $t\in \lbrack \tau ,T]$.
\end{lemma}

\begin{proof}
Let $u_{1},u_{2}$ be two solutions with $u_{i}\left( \tau \right) =u_{\tau }$%
, $i=1,2$, and such that $u_{i}\left( t\right) \geq 0$ for all $t$. Denote $%
w\left( t\right) =u_{1}\left( t\right) -u_{2}\left( t\right) $. Then in a
standard way by the Mean Value Theorem we obtain%
\begin{eqnarray*}
\frac{1}{2}\frac{d}{dt}\left\Vert w\left( t\right) \right\Vert ^{2} &\leq
&-\int_{\Omega }\left( f(t,u_{1}\left( t,x\right) )-f\left( t,u_{2}\left(
t,x\right) \right) ,w\left( t,x\right) \right) dx \\
&=&-\int_{\Omega }\left( f_{u}\left( t,v\left( t,x,u_{1},u_{2}\right)
\right) w\left( t,x\right) ,w\left( t,x\right) \right) dx \\
&\leq &C_{3}\left( t\right) \left\Vert w\left( t\right) \right\Vert ^{2},
\end{eqnarray*}%
where $v\left( t,x,u_{1},u_{2}\right) \in L\left( u_{1}\left( t,x\right)
,u_{2}\left( t,x\right) \right) \ =\{\alpha u_{1}\left( t,x\right) +\left(
1-\alpha \right) u_{2}\left( t,x\right) :\alpha \in \lbrack 0,1]\}$, so
that, $v\left( t,x,u_{1},u_{2}\right) \geq 0.$ The uniqueness follows from
Gronwall's lemma
\end{proof}

\bigskip

We prove also a result, which is similar to Lemma \ref{ConstantsP}. Denote
by $p_{i}^{j}$, $\alpha ^{j},\ C_{1}^{j}$, $C_{2}^{j}$ the constants
corresponding to $f_{j}$ in (\ref{Cond1})-(\ref{Cond2}). For problems (\ref%
{P1}) and (\ref{P2}), respectively. Arguing as in the proof of Lemma \ref%
{ConstantsP} we obtain the following lemma.

\begin{lemma}
\label{ConstantsPPositive}If $f_{j}$ satisfy (\ref{Cond1})-(\ref{Cond2}) and
(\ref{Ineqfh}) for $u\in \mathbb{R}_{d}^{+}$, then $p_{i}^{1}\geq p_{i}^{2}$
for all $i.$
\end{lemma}

\begin{theorem}
\label{ComparisonPositive}Let $f_{j},h_{j}$ satisfy (\ref{Cond3}) and (\ref%
{Cond10})-(\ref{Cond11}). Assume that $f_{j},h_{j}$ satisfy (\ref{Cond1})-(%
\ref{Cond2})\ and (\ref{Ineqfh}) for $u\in \mathbb{R}_{d}^{+}$. If $0\leq
u_{\tau }^{1}\leq u_{\tau }^{2}$ and we suppose that $f_{2}$ satisfies (\ref%
{Cond4}) for $u\in \mathbb{R}_{+}^{d}$ with $R_{0}^{2}\geq 2\max
\{K^{2}\left( \Vert u_{\tau }^{1}\Vert ,\tau ,T\right) ,K^{2}\left( \Vert
u_{\tau }^{2}\Vert ,\tau ,T\right) \}$, where $K\left( \Vert u_{\tau
}^{j}\Vert ,\tau ,T\right) $ is taken from (\ref{Est2}), we have $0\leq
u_{1}\left( t\right) \leq u_{2}\left( t\right) $, for all $t\in \lbrack \tau
,T]$, where $u_{1}\left( \text{\textperiodcentered }\right) ,u_{2}\left( 
\text{\textperiodcentered }\right) $ are the solutions corresponding to $%
u_{\tau }^{1}$ and $u_{\tau }^{2}$, respectively.
\end{theorem}

\begin{proof}
As the solutions $u_{1}\left( \text{\textperiodcentered }\right) $, $%
u_{2}\left( \text{\textperiodcentered }\right) $ corresponding to $u_{\tau
}^{1}$ and $u_{\tau }^{2}$ are non-negative, repeating exactly the same
steps of the proof of Theorem \ref{Comparison} we obtain the desired result.
\end{proof}

\begin{remark}
The results remains valid if, instead, $f_{1}$ satisfies (\ref{Cond4}) with
the same $R_{0}.$
\end{remark}

\begin{remark}
If $f_{2}$ satisfies (\ref{Cond4}) for an arbitrary $R_{0}>0$ (that is, in
the whole space $\mathbb{R}^{d}$), then the result is true for any initial
data $0\leq u_{\tau }^{1}\leq u_{\tau }^{2}.$
\end{remark}

We shall need also the following modification of Theorem \ref%
{ComparisonPositive}.

\begin{theorem}
\label{ComparisonPositive2}Let $f_{j},h_{j}$ satisfy (\ref{Cond3}) and (\ref%
{Cond10})-(\ref{Cond11}). Assume that $f_{j},h_{j}$ satisfy (\ref{Cond1})-(%
\ref{Cond2})\ and (\ref{Ineqfh}) for $u\in \mathbb{R}_{+}^{d}$. Let $0\leq
u_{\tau }^{1}\leq u_{\tau }^{2}.$ We suppose that $f_{2}$ satisfies 
\begin{equation}
f_{2}^{i}\left( t,u\right) \geq f_{2}^{i}\left( t,v\right) -\varepsilon ,
\label{Cond4B}
\end{equation}%
for any $t\in \lbrack \tau ,T]$ and any $u,v\in \mathbb{R}_{+}^{d}$ such
that $u^{i}=v^{i}$ and $u^{j}\leq v^{j}$ if $j\neq i,$ and $\left\vert
u\right\vert ,\left\vert v\right\vert \leq R_{0}$ with $R_{0}^{2}\geq 2\max
\{K^{2}\left( \Vert u_{\tau }^{1}\Vert ,\tau ,T\right) ,K^{2}\left( \Vert
u_{\tau }^{2}\Vert ,\tau ,T\right) \}$, where $K\left( \Vert u_{\tau
}^{j}\Vert ,\tau ,T\right) $ is taken from (\ref{Est2})

Then there exists a constant $C\left( \tau ,T\right) $ such that 
\begin{equation*}
\left\Vert \left( u_{1}\left( t\right) -u_{2}\left( t\right) \right)
^{+}\right\Vert \leq C\left( \tau ,T\right) \varepsilon \text{, for all }%
t\in \lbrack \tau ,T],
\end{equation*}%
where $u_{1}\left( \text{\textperiodcentered }\right) ,u_{2}\left( \text{%
\textperiodcentered }\right) $ are the solutions corresponding to $u_{\tau
}^{1}$ and $u_{\tau }^{2}$, respectively.
\end{theorem}

\begin{proof}
Arguing as in the proof of Theorem \ref{Comparison} we obtain the inequality%
\begin{eqnarray*}
&&\frac{1}{2}\frac{d}{dt}\left\Vert \left( u_{1}-u_{2}\right)
^{+}\right\Vert ^{2}+a\left\Vert \left( u_{1}-u_{2}\right) ^{+}\right\Vert
_{V}^{2} \\
&\leq &C_{3}\left( t\right) \left\Vert \left( u_{1}-u_{2}\right)
^{+}\right\Vert ^{2}-\int_{\Omega }\left( f_{2}\left( t,v_{2}\right)
-f_{2}\left( t,u_{2}\right) ,\left( u_{1}-u_{2}\right) ^{+}\right) dx,
\end{eqnarray*}%
where $v_{2}$ is defined in (\ref{v2}).

Using (\ref{IneqCond4}), $v_{2}\leq u_{2},$ $\left\vert v_{2}\right\vert
^{2}\leq \left\vert u_{1}\right\vert ^{2}+\left\vert u_{2}\right\vert ^{2}$,
(\ref{Est2}) and (\ref{Cond4B}) we get%
\begin{equation*}
f_{2}^{i}\left( t,v_{2}\right) -f_{2}^{i}\left( t,u_{2}\right) \geq
-\varepsilon .
\end{equation*}%
Thus%
\begin{eqnarray*}
\frac{d}{dt}\left\Vert \left( u_{1}-u_{2}\right) ^{+}\right\Vert ^{2} &\leq
&2C_{3}\left( t\right) \left\Vert \left( u_{1}-u_{2}\right) ^{+}\right\Vert
^{2}+2\varepsilon \int_{\Omega }\sum_{i\in J}\left(
u_{1}^{i}-u_{2}^{i}\right) ^{+}dx \\
&\leq &\left( 2C_{3}\left( t\right) +1\right) \left\Vert \left(
u_{1}-u_{2}\right) ^{+}\right\Vert ^{2}+K\varepsilon ^{2},
\end{eqnarray*}%
for some constant $K>0$. By Gronwall's lemma we get 
\begin{eqnarray*}
\left\Vert \left( u_{1}\left( t\right) -u_{2}\left( t\right) \right)
^{+}\right\Vert ^{2} &\leq &\left\Vert \left( u_{\tau }^{1}-u_{\tau
}^{2}\right) ^{+}\right\Vert ^{2}e^{\int_{\tau }^{t}\left( 2C_{3}\left(
s\right) +1\right) ds}+K\varepsilon ^{2}\int_{\tau
}^{t}e^{\int_{r}^{t}\left( 2C_{3}\left( s\right) +1\right) ds}dr \\
&\leq &C^{2}\left( \tau ,T\right) \varepsilon ^{2}.
\end{eqnarray*}
\end{proof}

\bigskip

Let us consider now the multivalued case. We will obtain first some
auxiliary statements.

We shall define suitable approximations. For any $n\geq 1$ we put $%
f_{n}^{i}(t,u)=\psi _{n}\left( \left\vert u\right\vert \right) f^{i}\left(
t,u\right) +\left( 1-\psi _{n}\left( \left\vert u\right\vert \right) \right)
g^{i}\left( t,u\right) $, where $g^{i}\left( t,u\right) =\left\vert
u^{i}\right\vert ^{p_{i}-2}u^{i}+f^{i}\left( t,0,...,0\right) $, and $\psi
_{n}$ was defined in (\ref{FiK}). Then $f_{n}\in \mathbb{C}(\mathbb{[\tau }%
,T]\times \mathbb{R}^{d};\mathbb{R}^{d})$ and for any $A>0,$ 
\begin{equation}
\sup\limits_{t\in \lbrack \tau ,T]}\sup\limits_{|u|\leq
A}|f_{n}(t,u)-f(t,u)|\rightarrow 0,\ as\ n\rightarrow \infty .  \label{Convf}
\end{equation}%
We will check first that $f_{n}$ satisfy conditions (\ref{Cond1})-(\ref%
{Cond2}) for $u\in \mathbb{R}_{+}^{d}$, where the constants do not depend on 
$n$.

\begin{lemma}
\label{Aux2}Let $f$ satisfy (\ref{Cond1})-(\ref{Cond2}) for $u\in \mathbb{R}%
_{+}^{d}$. For all $n\geq 1$ we have%
\begin{equation}
\sum_{i=1}^{d}|f_{n}^{i}(t,u)|^{\frac{p_{i}}{p_{i}-1}}\leq
D_{1}(1+\sum_{i=1}^{d}|u^{i}|^{p_{i}}),\ \left( f_{n}(t,u),u\right) \geq
\gamma \sum_{i=1}^{d}|u^{i}|^{p_{i}}-D_{2},  \label{AuxEst4}
\end{equation}%
for $u\in \mathbb{R}_{d}^{+}$, where the positive constants $D_{1}$, $D_{2}$%
, $\gamma $ do not depend on $n$.

if $\left\vert u\right\vert >n+1$, then for any $w\in \mathbb{R}^{d}$ we
have 
\begin{equation}
\left( f_{nu}(t,u)w,w\right) \geq 0.  \label{AuxEst5}
\end{equation}

Moreover, if $f,h$ satisfy (\ref{Cond11}), then $f_{n},h$ also satisfies
this condition.
\end{lemma}

\begin{proof}
In view of (\ref{Cond1})-(\ref{Cond2}) we get%
\begin{equation*}
\left( f_{n}(t,u),u\right) =\psi _{n}\left( \left\vert u\right\vert \right)
\left( f\left( t,u\right) ,u\right) +\left( 1-\psi _{n}\left( \left\vert
u\right\vert \right) \right) \left( g\left( u\right) ,u\right)
\end{equation*}%
\begin{equation*}
\geq \psi _{n}\left( \left\vert u\right\vert \right) \left( \alpha
\sum_{i=1}^{d}|u^{i}|^{p_{i}}-C_{2}\right) +\left( 1-\psi _{n}\left(
\left\vert u\right\vert \right) \right) \sum_{i=1}^{d}|u^{i}|^{p_{i}}+\left(
1-\psi _{n}\left( \left\vert u\right\vert \right) \right)
\sum_{i=1}^{d}f^{i}\left( t,0,...,0\right) u^{i}
\end{equation*}%
\begin{equation*}
\geq \psi _{n}\left( \left\vert u\right\vert \right) \alpha
\sum_{i=1}^{d}|u^{i}|^{p_{i}}-C_{2}+\left( 1-\psi _{n}\left( \left\vert
u\right\vert \right) \right) \frac{1}{2}\sum_{i=1}^{d}|u^{i}|^{p_{i}}-K_{1}%
\left( 1-\psi _{n}\left( \left\vert u\right\vert \right) \right)
\sum_{i=1}^{d}\left\vert f^{i}\left( t,0,...,0\right) \right\vert ^{\frac{%
p_{i}}{p_{i}-1}}
\end{equation*}%
\begin{equation*}
\geq \widetilde{\alpha }\sum_{i=1}^{d}|u^{i}|^{p_{i}}-C_{2}-K_{1}C_{1},
\end{equation*}%
where $\widetilde{\alpha }=\min \{\frac{1}{2},\alpha \}$, for some constant $%
K_{1}>0$. Also,%
\begin{equation*}
\sum_{i=1}^{d}|f_{n}^{i}(t,u)|^{\frac{p_{i}}{p_{i}-1}}\leq K_{2}\left(
\sum_{i=1}^{d}|f^{i}(t,u)|^{\frac{p_{i}}{p_{i}-1}}+\sum_{i=1}^{d}\left\vert
g^{i}\left( u\right) \right\vert ^{\frac{p_{i}}{p_{i}-1}}\right)
\end{equation*}%
\begin{equation*}
\leq K_{3}\left(
C_{1}(1+\sum_{i=1}^{d}|u^{i}|^{p_{i}})+\sum_{i=1}^{d}|u^{i}|^{p_{i}}+%
\sum_{i=1}^{d}\left\vert f^{i}\left( t,0,...,0\right) \right\vert ^{\frac{%
p_{i}}{p_{i}-1}}\right)
\end{equation*}%
\begin{equation*}
\leq K_{4}\left( \sum_{i=1}^{d}|u^{i}|^{p_{i}}+1\right) ,
\end{equation*}%
for some constant $K_{4}>0.$ Thus, for $D_{1}=K_{4}$, $%
D_{2}=C_{2}+K_{1}C_{1} $, $\widetilde{\alpha }=\min \{\frac{1}{2},\alpha \}$
we have (\ref{AuxEst4}) for the functions $f_{n}$.

Moreover, if $\left\vert u\right\vert >n+1$, then for any $w\in \mathbb{R}%
^{d},$%
\begin{equation}
\left( f_{nu}(t,u)w,w\right) =\left( g_{u}(t,u)w,w\right)
=\sum_{i=1}^{d}\left( p_{i}-1\right) \left\vert u^{i}\right\vert
^{p_{i}-2}w_{i}^{2}\geq 0.  \label{IneqPositive}
\end{equation}

Finally, if (\ref{Cond11}) is satisfied, then%
\begin{align*}
& h^{i}\left( t,x\right) -f_{n}^{i}\left( t,u\right) \\
& =\psi _{n}\left( \left\vert u\right\vert \right) \left( h^{i}\left(
t,x\right) -f^{i}\left( t,u\right) \right) +\left( 1-\psi _{n}\left(
\left\vert u\right\vert \right) \right) \left( h^{i}\left( t,x\right)
-f^{i}\left( t,0,...,0\right) \right) \geq 0,
\end{align*}%
for all $i$, a.e. $\left( t,x\right) \in \left( \tau ,T\right) \times \Omega 
$ and $u$ such that $u^{i}=0$ and $u^{j}\geq 0$ if $j\neq i$.
\end{proof}

\bigskip

Let $2\leq q_{i}\leq p_{i}$, $i=1,...,d$. We define also the following
approximations $\widetilde{f}_{n}^{i}(t,u)=\psi _{n}\left( \left\vert
u\right\vert \right) f^{i}\left( t,u\right) +\left( 1-\psi _{n}\left(
\left\vert u\right\vert \right) \right) \widetilde{g}^{i}\left( t,u\right) $%
, where $\widetilde{g}^{i}\left( t,u\right) =\left\vert u^{i}\right\vert
^{p_{i}-2}u^{i}+\left\vert u^{i}\right\vert ^{q_{i}-2}u^{i}+f^{i}\left(
t,0,...,0\right) $. Then (\ref{Convf}) holds. We check that $\widetilde{f}%
_{n}$ satisfy conditions (\ref{Cond1})-(\ref{Cond2}) for $u\in \mathbb{R}%
_{+}^{d}$, where the constants do not depend on $n$.

\begin{lemma}
\label{Aux2B}Let $f$ satisfy (\ref{Cond1})-(\ref{Cond2}) for $u\in \mathbb{R}%
_{+}^{d}$. For all $n\geq 1$ we have%
\begin{equation}
\sum_{i=1}^{d}|\widetilde{f}_{n}^{i}(t,u)|^{\frac{p_{i}}{p_{i}-1}}\leq
D_{1}(1+\sum_{i=1}^{d}|u^{i}|^{p_{i}}),\ \left( \widetilde{f}%
_{n}(t,u),u\right) \geq \gamma \sum_{i=1}^{d}|u^{i}|^{p_{i}}-D_{2},
\label{AuxEst4B}
\end{equation}%
for $u\in \mathbb{R}_{d}^{+}$, where the positive constants $D_{1}$, $D_{2}$%
, $\gamma $ do not depend on $n$.

if $\left\vert u\right\vert >n+1$, then for any $w\in \mathbb{R}^{d}$ we
have 
\begin{equation}
\left( \widetilde{f}_{nu}(t,u)w,w\right) \geq 0.  \label{AuxEst5B}
\end{equation}

Moreover, if $f,h$ satisfy (\ref{Cond11}), then $\widetilde{f}_{n},h$ also
satisfies this condition.
\end{lemma}

\begin{proof}
In view of (\ref{AuxEst4}) we have%
\begin{equation*}
\left( \widetilde{f}_{n}(t,u),u\right) =\left( f_{n}(t,u),u\right) +\left(
1-\psi _{n}\left( \left\vert u\right\vert \right) \right)
\sum_{i=1}^{d}\left\vert u^{i}\right\vert ^{q_{i}}\geq \gamma
\sum_{i=1}^{d}|u^{i}|^{p_{i}}-D_{2},
\end{equation*}%
\begin{eqnarray*}
\sum_{i=1}^{d}|\widetilde{f}_{n}^{i}(t,u)|^{\frac{p_{i}}{p_{i}-1}} &\leq
&K_{1}\left( \sum_{i=1}^{d}|f_{n}^{i}(t,u)|^{\frac{p_{i}}{p_{i}-1}%
}+\sum_{i=1}^{d}|u^{i}|^{\frac{p_{i}(q_{i}-1)}{p_{i}-1}}\right) \\
&\leq &K_{1}\left(
D_{1}(1+\sum_{i=1}^{d}|u^{i}|^{p_{i}})+\sum_{i=1}^{d}|u^{i}|^{q_{i}}\right)
\leq K_{2}\left( 1+\sum_{i=1}^{d}|u^{i}|^{p_{i}}\right) ,
\end{eqnarray*}%
where we have used that $p_{i}\geq q_{i}$ implies $\frac{p_{i}}{p_{i}-1}\leq 
\frac{q_{i}}{q_{i}-1}$. Finally, (\ref{AuxEst5B}) and condition (\ref{Cond11}%
) are proved in the same way as in Lemma \ref{Aux2}.
\end{proof}

\bigskip

For every $n\geq 1$ consider the sequence $f_{n}^{\varepsilon }\left(
t,u\right) $ defined by $f_{n}^{\epsilon }(t,u)=\int_{\mathbb{R}^{d}}\rho
_{\epsilon }(s)b_{n}(t,u-s)ds$, where either $b_{n}=f_{n}$ or $b_{n}=%
\widetilde{f}_{n}$, defined before. Since any $b_{n}$ are uniformly
continuous on $[\tau ,T]\times \lbrack -k-1,k+1]$, for any $k\geq 1$, there
exist $\epsilon _{k,n}\in (0,1)$ such that for all $u$ satisfying $|u|\leq k$%
, and for all $s$ for which $|u-s|<\epsilon _{k,n}$ we have 
\begin{equation*}
\sup\limits_{t\in \lbrack \tau ,T]}|b_{n}(t,u)-b_{n}(t,s)|\leq \frac{1}{k}.
\end{equation*}%
We put $f_{n}^{k}(t,u)=f_{n}^{\epsilon _{k,n}}(t,u)$. Then $%
f_{n}^{k}(t,\cdot )\in \mathbb{C}^{\infty }(\mathbb{R}^{d};\mathbb{R}^{d})$,
for all $t\in \lbrack \tau ,T],\ k,n\geq 1$. Since for any compact subset $%
A\subset \mathbb{R}^{d}$ and any $n$ we have $f_{n}^{k}\rightarrow b_{n}$
uniformly on $[\tau ,T]\times A,$ we obtain the existence of a sequence $%
\delta _{nk}\in \left( 0,1\right) $ such that $\delta _{nk}\rightarrow 0$,
as $k\rightarrow \infty $, and $\left\vert f_{n}^{ki}\left( t,u\right)
-b_{n}^{i}\left( t,u\right) \right\vert \leq \delta _{nk},$ for any $i,$ $n$
and any $u$ satisfying $\left\vert u\right\vert \leq n+2$. We define the
function $F_{n}^{k}=\left( F_{n}^{k1},...,F_{n}^{kd}\right) $ given by 
\begin{equation}
F_{n}^{ki}\left( t,u\right) =f_{n}^{ki}\left( t,u\right) -p\delta _{nk},
\label{Fn}
\end{equation}%
where $p\in \mathbb{N}.$

\begin{lemma}
\label{Aux3}Let $f$ satisfy (\ref{Cond1})-(\ref{Cond2}) for $u\in \mathbb{R}%
_{+}^{d}$. For all $n,k\geq 1$ we have%
\begin{equation}
\sum_{i=1}^{d}|F_{n}^{ki}(t,u)|^{\frac{p_{i}}{p_{i}-1}}\leq
D_{3}(1+\sum_{i=1}^{d}|u^{i}|^{p_{i}}),\ \left( F_{n}^{k}(t,u),u\right) \geq
\nu \sum_{i=1}^{d}|u^{i}|^{p_{i}}-D_{4},  \label{AuxEst6}
\end{equation}%
for $u\in \mathbb{R}_{+}^{d}$, where the positive constants $D_{3}$, $D_{4},$
$\nu $ do not depend neither on $n$ nor $k$.

Moreover, if $f,h$ satisfy (\ref{Cond11}), then $F_{n}^{k},h$ also satisfy
this condition if $\left\vert u\right\vert \leq n+2$.
\end{lemma}

\begin{proof}
Since $f_{n}$ satisfy (\ref{AuxEst4}) and $\widetilde{f}_{n}$ satisfies (\ref%
{AuxEst4B}), we have%
\begin{equation*}
\sum_{i=1}^{d}|F_{n}^{ki}(t,u)|^{\frac{p_{i}}{p_{i}-1}}\leq R_{1}\left(
\sum_{i=1}^{d}|f_{n}^{ki}(t,u)|^{\frac{p_{i}}{p_{i}-1}}+1\right)
\end{equation*}%
\begin{equation*}
\leq R_{1}\left( \sum_{i=1}^{d}\left( \left( \int_{\mathbb{R}^{d}}\rho
_{\epsilon _{k}}(s)ds\right) ^{\frac{1}{p_{i}-1}}\int_{\mathbb{R}^{d}}\rho
_{\epsilon _{k}}(s)\left\vert b_{n}^{i}(t,u-s)\right\vert ^{\frac{p_{i}}{%
p_{i}-1}}ds\right) +1\right)
\end{equation*}%
\begin{equation*}
\leq R_{2}\left( \sum_{i=1}^{d}\int_{\mathbb{R}^{d}}\rho _{\epsilon
_{k}}(s)(1+|u^{i}-s^{i}|^{p_{i}})ds+1\right) \leq R_{3}\left(
\sum_{i=1}^{d}\int_{\mathbb{R}^{d}}\rho _{\epsilon
_{k}}(s)(|u^{i}|^{p_{i}}+\epsilon _{k}^{p_{i}})ds+1\right)
\end{equation*}%
\begin{equation*}
\leq R_{4}(1+\sum_{i=1}^{d}|u^{i}|^{p_{i}}),
\end{equation*}%
for some constant $R_{4}>0.$

On the other hand, 
\begin{equation*}
\left( F_{n}^{k}(t,u),u\right) =\int_{\mathbb{R}^{d}}\rho _{\epsilon
_{k}}(s)(b_{n}(t,u-s),u-s)ds+\int_{\mathbb{R}^{d}}\rho _{\epsilon
_{k}}(s)\left( b_{n}(t,u-s),s\right) ds-p\delta _{nk}\sum_{i=1}^{n}u^{i}
\end{equation*}%
\begin{equation*}
\geq \int_{\mathbb{R}^{d}}\rho _{\epsilon _{k}}(s)(\gamma
\sum_{i=1}^{d}\left\vert u^{i}-s^{i}\right\vert ^{p_{i}}-D_{2})ds-\int_{%
\mathbb{R}^{d}}\rho _{\epsilon _{k}}(s)\sum_{i=1}^{d}\left( \frac{\gamma }{%
2D_{1}}\left\vert b_{n}^{i}\left( t,u-s\right) \right\vert ^{\frac{p_{i}}{%
p_{i}-1}}+R_{5}\left\vert s^{i}\right\vert ^{p_{i}}\right) ds
\end{equation*}%
\begin{equation*}
-p\delta _{nk}\sum_{i=1}^{n}u^{i}\geq \int_{\mathbb{R}^{d}}\rho _{\epsilon
_{k}}(s)(\gamma \sum_{i=1}^{d}\left\vert u^{i}-s^{i}\right\vert
^{p_{i}}-D_{2})ds-\int_{\mathbb{R}^{d}}\rho _{\epsilon _{k}}(s)\left( \frac{%
\gamma }{2}\sum_{i=1}^{d}|u^{i}-s^{i}|^{p_{i}}+R_{6}\right) ds
\end{equation*}%
\begin{equation*}
-p\delta _{nk}\sum_{i=1}^{n}u^{i}\geq \frac{\gamma }{2}\int_{\mathbb{R}%
^{d}}\rho _{\epsilon _{k}}(s)\sum_{i=1}^{d}\left\vert u^{i}-s^{i}\right\vert
^{p_{i}}ds-R_{7}-p\delta _{nk}\sum_{i=1}^{n}u^{i}\geq \nu
\sum_{i=1}^{d}\left\vert u^{i}\right\vert ^{p_{i}}-R_{8},
\end{equation*}%
for some constants $\nu ,R_{8}>0$, where in the last inequality we have used
that for some $D>0,$%
\begin{equation*}
\left\vert u^{i}\right\vert ^{p_{i}}=\left\vert u^{i}-s^{i}+s^{i}\right\vert
^{p_{i}}\leq D\left( \left\vert u^{i}-s^{i}\right\vert ^{p_{i}}+\left\vert
s^{i}\right\vert ^{p_{i}}\right) \leq D\left( \left\vert
u^{i}-s^{i}\right\vert ^{p_{i}}+\epsilon _{k}^{p_{i}}\right) .
\end{equation*}%
Hence, (\ref{AuxEst6}) holds.

In view of Lemmas \ref{Aux2}, \ref{Aux2B} the functions $b_{n},h$ satisfy (%
\ref{Cond11}). Hence, $\left\vert f_{n}^{ki}\left( t,u\right)
-b_{n}^{i}\left( t,u\right) \right\vert \leq \delta _{nk},$ for any $i,$ $n$
and any $u$ satisfying $\left\vert u\right\vert \leq n+2,$ implies that%
\begin{eqnarray*}
h^{i}\left( t,x\right) -F_{n}^{ki}\left( t,u\right) &=&h^{i}\left(
t,x\right) -f_{n}^{ki}\left( t,u\right) +p\delta _{nk} \\
&\geq &h^{i}\left( t,x\right) -b_{n}^{i}\left( t,u\right) \geq 0,
\end{eqnarray*}%
for $u$ such that $u_{i}=0$, $u_{j}\geq 0$, $j\neq i$, and $\left\vert
u\right\vert \leq n+2$.
\end{proof}

\bigskip

Define a smooth function $\phi _{n}:\mathbb{R}_{+}\rightarrow \lbrack 0,1]$
satisfying%
\begin{equation*}
\phi _{n}\left( s\right) =\left\{ 
\begin{array}{c}
1\text{, if }0\leq s\leq n+1+\overline{\gamma }, \\ 
0\leq \phi _{n}\left( s\right) \leq 1\text{, if }n+1+\overline{\gamma }\leq
s\leq n+2, \\ 
0\text{, if }s\geq n+2,%
\end{array}%
\right.
\end{equation*}%
where $0<\overline{\gamma }<1$ is fixed. Let $l_{n}^{k}(t,u)$ be given by 
\begin{equation*}
l_{n}^{k}(t,u)=\phi _{n}\left( \left\vert u\right\vert \right)
F_{n}^{k}\left( t,u\right) +\left( 1-\phi _{n}\left( \left\vert u\right\vert
\right) \right) b_{n}\left( t,u\right) .
\end{equation*}%
Since for any compact subset $A\subset \mathbb{R}^{d}$ and any $n$ we have $%
f_{n}^{k}\rightarrow b_{n}$ uniformly on $[\tau ,T]\times A$ as $%
k\rightarrow \infty $ it is clear that 
\begin{equation*}
\sup\limits_{t\in \lbrack \tau ,T]}\sup\limits_{|u|\leq
A}|l_{n}^{k}(t,u)-b_{n}(t,u)|\rightarrow 0,\ as\ k\rightarrow \infty .
\end{equation*}

\begin{lemma}
\label{Aux4}Let $f$ satisfy (\ref{Cond1})-(\ref{Cond2}) for $u\in \mathbb{R}%
_{+}^{d}$. For all $n,k\geq 1$ we have%
\begin{equation}
\sum_{i=1}^{d}|l_{n}^{ki}(t,u)|^{\frac{p_{i}}{p_{i}-1}}\leq
D_{5}(1+\sum_{i=1}^{d}|u^{i}|^{p_{i}}),\ \left( l_{n}^{k}(t,u),u\right) \geq
\lambda \sum_{i=1}^{d}|u^{i}|^{p_{i}}-D_{6},  \label{AuxEst7}
\end{equation}%
for $u\in \mathbb{R}_{+}^{d}$, where the positive constants $D_{5}$, $D_{6},$
$\lambda $ do not depend neither on $n$ nor $k$. Also, 
\begin{equation}
\left( l_{nu}^{k}(t,u)w,w\right) \geq -D_{7}(k,n)\left\vert w\right\vert
^{2},\text{ }\forall u,w,  \label{AuxEst8}
\end{equation}%
where $D_{7}(k,n)$ is a non-negative number.

Moreover, if $f,h$ satisfy (\ref{Cond11}), then $l_{n}^{k},h$ also satisfies
this condition.
\end{lemma}

\begin{proof}
The inequalities given in (\ref{AuxEst7}) are an easy consequence of (\ref%
{AuxEst4}), (\ref{AuxEst4B}) and (\ref{AuxEst6}).

On the other hand, if $u$ is such that $u_{i}=0$, $u_{j}\geq 0$, $j\neq i,$
then in view of Lemmas \ref{Aux2}, \ref{Aux2B} and \ref{Aux3} we have 
\begin{equation*}
\phi _{n}\left( \left\vert u\right\vert \right) \left( h^{i}\left(
t,x\right) -F_{n}^{ki}\left( t,u\right) \right) \geq 0,\ \left( 1-\phi
_{n}\left( \left\vert u\right\vert \right) \right) \left( h^{i}\left(
t,x\right) -b_{n}^{i}\left( t,u\right) \right) \geq 0,
\end{equation*}%
as $\phi _{n}\left( \left\vert u\right\vert \right) =0,$ for $\left\vert
u\right\vert \geq n+2$. Hence, $h^{i}\left( t,x\right) -l_{n}^{ki}\left(
t,u\right) \geq 0$, so that condition (\ref{Cond11}) holds.

It is also clear that $l_{n}^{k}\left( t,u\right) $ is continuously
differentiable with respect to $u$ for any $t$ and $u$. We obtain the
existence of $D_{7}\left( k,n\right) $ such that (\ref{AuxEst8}) holds.
Indeed, if $\left\vert u\right\vert \leq n+1+\overline{\gamma }$, then $%
l_{n}^{k}(t,u)=F_{n}^{k}\left( t,u\right) $, so that%
\begin{eqnarray}
\left\vert (l_{nu}^{k}(t,u)w,w)\right\vert &=&\left\vert
(F_{nu}^{k}(t,u)w,w)\right\vert =\left\vert (f_{nu}^{k}(t,u)w,w)\right\vert
\label{IneqDeriv} \\
&\leq &\left\vert w\right\vert ^{2}\int_{\mathbb{R}^{d}}\left\vert \nabla
\rho _{\epsilon _{k}}\left( u-s\right) \right\vert \left\vert b_{n}\left(
s,u\right) \right\vert ds\leq R_{1}\left( k,n\right) \left\vert w\right\vert
^{2}.  \notag
\end{eqnarray}

If $\left\vert u\right\vert \geq n+2$, then $l_{n}^{k}(t,u)=b_{n}\left(
t,u\right) $, so that by (\ref{AuxEst5}), (\ref{AuxEst5B}) we have $%
(l_{nu}^{k}(t,u)w,w)=(b_{nu}(t,u)w,w)\geq 0.$ Finally, if $n+1+\overline{%
\gamma }<\left\vert u\right\vert <n+2,$ we have%
\begin{align*}
(l_{nu}^{k}(t,u)w,w)& =\phi _{n}\left( \left\vert u\right\vert \right)
(F_{nu}^{k}(t,u)w,w)+\left( 1-\phi _{n}\left( \left\vert u\right\vert
\right) \right) (b_{nu}(t,u)w,w) \\
& +\sum_{i,j=1}^{d}\frac{\partial }{\partial u_{j}}\phi _{n}\left(
\left\vert u\right\vert \right) F_{n}^{ki}(t,u)w_{i}w_{j}-\sum_{i,j=1}^{d}%
\frac{\partial }{\partial u_{j}}\phi _{n}\left( \left\vert u\right\vert
\right) b_{n}^{i}(t,u)w_{i}w_{j} \\
& \geq -R_{2}\left( k,n\right) \left\vert w\right\vert ^{2},
\end{align*}%
where we have used similar arguments as in (\ref{IneqDeriv}), (\ref{AuxEst5}%
), (\ref{AuxEst5B}) and also that 
\begin{equation*}
\left\vert \frac{\partial }{\partial u_{j}}\phi _{n}\left( \left\vert
u\right\vert \right) \right\vert \leq R_{3}\left( n\right) ,\left\vert
b_{n}^{i}(t,u)\right\vert \leq R_{4}\left( n\right) ,\left\vert
F_{n}^{ki}(t,u)\right\vert \leq R_{5}\left( k,n\right) ,
\end{equation*}%
for any $u$ satisfying $n+1+\overline{\gamma }<\left\vert u\right\vert <n+2$%
, $t\in \lbrack \tau ,T]$ and any $i,j$.
\end{proof}

\bigskip

Now we are ready to obtain the weak comparison principle for positive
solutions.

\begin{theorem}
\label{ComparisonPositiveGeneral}Let $f_{j},h_{j}$ satisfy (\ref{Cond10})-(%
\ref{Cond11}). Assume that $f_{j},h_{j}$ satisfy (\ref{Cond1})-(\ref{Cond2}%
)\ and (\ref{Ineqfh}) for $u\in \mathbb{R}_{+}^{d}$. We suppose that either $%
f_{1}$ or $f_{2}$ satisfies (\ref{Cond4}) for $u\in \mathbb{R}_{+}^{d}$ and
for an arbitrary $R_{0}>0$. If $0\leq u_{\tau }^{1}\leq u_{\tau }^{2},$
there exist two solutions $u_{1},u_{2}$ (of (\ref{P1}) and (\ref{P2}),
respectively, with $u_{1}\left( \tau \right) =u_{\tau }^{1}$, $u_{2}\left(
\tau \right) =u_{\tau }^{2},$ such that $0\leq u_{1}\left( t\right) \leq
u_{2}\left( t\right) $, for all $t\in \lbrack \tau ,T].$
\end{theorem}

\begin{proof}
For $f_{j}$ let us consider the approximations $\widetilde{f}_{1n},$ $f_{2n}$
defined before with $q_{i}=p_{i}^{2}\leq p_{i}^{1}$, where the last
inequality follows from Lemma \ref{ConstantsPPositive}. Then by Lemmas \ref%
{Aux2}, \ref{Aux2B} we have that $\widetilde{f}_{1n},f_{2n}$ satisfy (\ref%
{Cond1})-(\ref{Cond2}) for $u\in \mathbb{R}_{+}^{d}$ with constants not
depending on $n$. Also, (\ref{AuxEst5}), (\ref{AuxEst5B}) hold and (\ref%
{Cond11}) is satisfied in both \ cases. Moreover, by (\ref{Ineqfh}) for $%
u\in \mathbb{R}_{+}^{d}$, we have%
\begin{eqnarray}
\widetilde{f}_{1n}^{i}\left( t,u\right) &=&\psi _{n}\left( \left\vert
u\right\vert \right) f_{1}^{i}\left( t,u\right) +\left( 1-\psi _{n}\left(
\left\vert u\right\vert \right) \right) \left( \left\vert u^{i}\right\vert
^{p_{i}^{1}-1}+\left\vert u^{i}\right\vert ^{p_{i}^{2}-1}+f_{1}^{i}\left(
t,0,...,0\right) \right)  \label{Ineqfn} \\
&\geq &\psi _{n}\left( \left\vert u\right\vert \right) f_{2}^{i}\left(
t,u\right) +\left( 1-\psi _{n}\left( \left\vert u\right\vert \right) \right)
\left( \left\vert u^{i}\right\vert ^{p_{i}^{2}-1}+f_{2}^{i}\left(
t,0,...,0\right) \right) =f_{2n}^{i}\left( t,u\right) ,  \notag
\end{eqnarray}%
if $u\in \mathbb{R}_{+}^{d}.$

As explained before, we can choose a sequence $\delta _{nk}\in \left(
0,1\right) $ such that $\delta _{nk}\rightarrow 0$, as $k\rightarrow \infty $%
, and $\left\vert f_{1n}^{ki}\left( t,u\right) -\widetilde{f}_{1n}^{i}\left(
t,u\right) \right\vert \leq \delta _{nk},$ $\left\vert f_{2n}^{ki}\left(
t,u\right) -f_{2n}^{i}\left( t,u\right) \right\vert \leq \delta _{nk},$ for
any $i,$ $n$ and any $u$ satisfying $\left\vert u\right\vert \leq n+2$.
Further, we consider the functions 
\begin{eqnarray*}
F_{1n}^{ki}\left( t,u\right) &=&f_{1n}^{ki}\left( t,u\right) -\delta _{nk},
\\
F_{2n}^{ki}\left( t,u\right) &=&f_{2n}^{ki}\left( t,u\right) -3\delta _{nk}.
\end{eqnarray*}%
By Lemma \ref{Aux3} we know that $F_{jn}^{k}$ satisfy (\ref{Cond1})-(\ref%
{Cond2}) for $u\in \mathbb{R}_{d}^{+}$ with constants not depending neither
on $k$ nor $n$, and condition (\ref{Cond11}) for $\left\vert u\right\vert
\leq n+2$, as well. Moreover, by (\ref{Ineqfn}) we have%
\begin{eqnarray}
F_{2n}^{ki}\left( t,u\right) &=&f_{2n}^{ki}\left( t,u\right) -3\delta _{nk}
\label{IneqFn} \\
&\leq &f_{2n}^{i}\left( t,u\right) -2\delta _{nk}\leq f_{1n}^{i}\left(
t,u\right) -2\delta _{nk}  \notag \\
&\leq &f_{1n}^{ki}\left( t,u\right) -\delta _{nk}=F_{1n}^{ki}\left(
t,u\right) ,  \notag
\end{eqnarray}%
if $u\in \mathbb{R}_{+}^{d}$ and $\left\vert u\right\vert \leq n+2.$

Suppose, for example, that $f_{2}$ satisfies condition (\ref{Cond4}) for $%
u\in \mathbb{R}_{+}^{d}$. For any $t\in \lbrack \tau ,T]$ and any $u,v\in 
\mathbb{R}_{+}^{d}$ such that $u^{i}=v^{i}$ and $u^{j}\leq v^{j}$ if $j\neq
i $, $\left\vert u\right\vert ,\left\vert v\right\vert \leq n$, we have%
\begin{eqnarray}
F_{2n}^{ki}\left( t,u\right) &=&f_{2n}^{ki}\left( t,u\right) -3\delta
_{nk}\geq f_{2n}\left( t,u\right) -4\delta _{nk}  \label{Cond4BF2} \\
&=&f_{2}^{i}\left( t,u\right) -4\delta _{nk}\geq f_{2}^{i}\left( t,v\right)
-4\delta _{nk}  \notag \\
&=&f_{2n}^{i}\left( t,v\right) -4\delta _{nk}\geq f_{2n}^{ki}\left(
t,v\right) -5\delta _{nk}=F_{2n}^{ki}\left( t,v\right) -2\delta _{nk}. 
\notag
\end{eqnarray}%
Hence, (\ref{Cond4B}) is satisfied with $R_{0}=n$ and $\varepsilon =2\delta
_{nk}.$

Now, we will define the functions 
\begin{eqnarray*}
l_{1n}^{k}(t,u) &=&\phi _{n}\left( \left\vert u\right\vert \right)
F_{1n}^{k}\left( t,u\right) +\left( 1-\phi _{n}\left( \left\vert
u\right\vert \right) \right) \widetilde{f}_{1n}\left( t,u\right) ,\text{ } \\
l_{2n}^{k}(t,u) &=&\phi _{n}\left( \left\vert u\right\vert \right)
F_{2n}^{k}\left( t,u\right) +\left( 1-\phi _{n}\left( \left\vert
u\right\vert \right) \right) f_{2n}\left( t,u\right) .
\end{eqnarray*}%
By Lemma \ref{Aux4} these functions satisfy (\ref{Cond1})-(\ref{Cond2}) for $%
u\in \mathbb{R}_{+}^{d}$ with constants not depending neither on $k$ nor $n$%
, inequality (\ref{AuxEst8}), and condition (\ref{Cond11}).

In view of (\ref{Ineqfn}), (\ref{IneqFn}) and $\phi _{n}\left( \left\vert
u\right\vert \right) =0$, if $\left\vert u\right\vert \geq n+2,$ we obtain 
\begin{equation*}
\phi _{n}\left( \left\vert u\right\vert \right) F_{1n}^{ki}\left( t,u\right)
\geq \phi _{n}\left( \left\vert u\right\vert \right) F_{2n}^{ki}\left(
t,u\right) ,\ \left( 1-\phi _{n}\left( \left\vert u\right\vert \right)
\right) \widetilde{f}_{1n}^{i}\left( t,u\right) \geq \left( 1-\phi
_{n}\left( \left\vert u\right\vert \right) \right) f_{2n}^{i}\left(
t,u\right)
\end{equation*}%
and then if $u\in \mathbb{R}_{+}^{d},$%
\begin{equation}
l_{1n}^{k}(t,u)\geq l_{2n}^{k}(t,u).  \label{Ineqlnk}
\end{equation}

On the other hand, since $l_{2n}^{k}(t,u)=F_{2n}^{k}\left( t,u\right) $ if $%
\left\vert u\right\vert \leq n+1+\overline{\gamma }$, (\ref{Cond4BF2})
implies 
\begin{equation}
l_{2n}^{k}(t,u)\geq l_{2n}^{k}(t,v)-2\delta _{nk},  \label{Cond4Bl2}
\end{equation}%
for any $t\in \lbrack \tau ,T]$ and any $u,v\in \mathbb{R}_{+}^{d}$ such
that $u^{i}=v^{i}$ and $u^{j}\leq v^{j}$ if $j\neq i,$ $\left\vert
u\right\vert ,\left\vert v\right\vert \leq n.$ Thus, (\ref{Cond4B}) is
satisfied with $R_{0}=n$ and $\varepsilon =2\delta _{nk}.$

We consider now the problems%
\begin{equation}
\left\{ 
\begin{array}{l}
\dfrac{\partial u}{\partial t}-a\Delta u+l_{jn}^{k}(t,u)=h_{j}\left(
t,x\right) ,\ \ (t,x)\in (\tau ,T)\times \Omega , \\ 
u|_{x\in \partial \Omega }=0, \\ 
u|_{t=\tau }=u_{\tau },%
\end{array}%
\right.  \label{Pnk}
\end{equation}%
and%
\begin{equation}
\left\{ 
\begin{array}{l}
\dfrac{\partial u}{\partial t}-a\Delta u+b_{jn}(t,u)=h_{j}\left( t,x\right)
,\ \ (t,x)\in (\tau ,T)\times \Omega , \\ 
u|_{x\in \partial \Omega }=0, \\ 
u|_{t=\tau }=u_{\tau },%
\end{array}%
\right.  \label{Pn}
\end{equation}%
where $b_{1n}=\widetilde{f}_{1n}$ and $b_{2n}=f_{2n}.$

In view of Lemma \ref{Aux4} and \cite[Lemma 5]{KapVal09} problem (\ref{Pnk})
has a unique weak solution $u_{n}^{k}\left( \text{\textperiodcentered }%
\right) $ such that $u_{n}^{k}\left( t\right) \geq 0$ for all $t\in \lbrack
\tau ,T]$. Let $u_{jn}^{k}\left( \text{\textperiodcentered }\right) $ be the
solutions of (\ref{Pnk}) corresponding to the initial data $u_{\tau }^{j}$,
where $j=1,2.$ Using Lemma \ref{Aux4} it is standard to obtain estimate (\ref%
{Est2}) with a constant $D$ not depending neither on $n$ nor $k$. Then the
solutions $u_{jn}^{k}\left( \text{\textperiodcentered }\right) $ of (\ref%
{Pnk}) satisfy%
\begin{equation*}
\left\Vert u_{jn}^{k}\left( t\right) \right\Vert ^{2}\leq \Vert u_{\tau
}^{j}\Vert ^{2}+D\int\limits_{\tau }^{T}(\Vert h_{j}(r)\Vert
^{2}+1)dr=K^{2}\left( \Vert u_{\tau }^{j}\Vert ,\tau ,T\right) .
\end{equation*}%
Thus, since $l_{2n}^{k}$ satisfy condition (\ref{Cond4B}) with $%
R_{0}^{2}=n^{2}\geq 2\max \{K^{2}\left( \Vert u_{\tau }^{1}\Vert ,\tau
,T\right) ,K^{2}\left( \Vert u_{\tau }^{2}\Vert ,\tau ,T\right) \}$, by
Theorem \ref{ComparisonPositive2} we know that as $u_{\tau }^{1}\leq u_{\tau
}^{2},$ we have 
\begin{equation}
\left\Vert \left( u_{1n}^{k}\left( t\right) -u_{2n}^{k}\left( t\right)
\right) ^{+}\right\Vert \leq 2C\left( \tau ,T\right) \delta _{nk}\text{, }
\label{IneqDeltank}
\end{equation}%
for all $t\in \lbrack \tau ,T]$, all $k$ and $n\geq \left( 2\max
\{K^{2}\left( \Vert u_{\tau }^{1}\Vert ,\tau ,T\right) ,K^{2}\left( \Vert
u_{\tau }^{2}\Vert ,\tau ,T\right) \}\right) ^{\frac{1}{2}}$.

Arguing as in the proof of Theorem \ref{ComparisonWeakScalar} we obtain that
for $j=1,2$ the sequence $u_{jn}^{k}$ converges (up to a subsequence) in the
sense of (\ref{Conv1})-(\ref{Conv7}) to a solution $u_{jn}$ of problem (\ref%
{Pn}) with initial data $u_{\tau }^{j}.$ In particular, as $k\rightarrow
\infty $ we have%
\begin{equation}
u_{jn}^{k}\left( t\right) \rightarrow u_{jn}\left( t\right) \text{ in }H%
\text{ for all }t\in \lbrack \tau ,T].  \label{ConvHnk}
\end{equation}%
Fix $n\geq \left( 2\max \{K^{2}\left( \Vert u_{\tau }^{1}\Vert ,\tau
,T\right) ,K^{2}\left( \Vert u_{\tau }^{2}\Vert ,\tau ,T\right) \}\right) ^{%
\frac{1}{2}}$, $i\in \{1,...,d\}$ and take any $t\in \lbrack \tau ,T].$
Denote by $\Omega _{nk}^{it}$ the set 
\begin{equation*}
\Omega _{nk}^{it}=\{x\in \Omega :u_{1n}^{ki}\left( t,x\right)
-u_{2n}^{ki}\left( t,x\right) \geq 0\}.
\end{equation*}%
By (\ref{IneqDeltank}) as $k\rightarrow \infty $ we have 
\begin{equation}
\int_{\Omega _{nk}^{t}}\left( u_{1n}^{ki}\left( t,x\right)
-u_{2n}^{ki}\left( t,x\right) \right) ^{2}dx=\int_{\Omega }\left( \left(
u_{1n}^{ki}\left( t\right) -u_{2n}^{ki}\left( t\right) \right) ^{+}\right)
^{2}dx\leq 4C^{2}\left( \tau ,T\right) \delta _{nk}^{2}\rightarrow 0.
\label{ConvOmegank}
\end{equation}%
Hence, (\ref{ConvHnk}) implies%
\begin{equation*}
\int_{\Omega \backslash \Omega _{nk}^{t}}\left( u_{1n}^{ki}\left( t,x\right)
-u_{2n}^{ki}\left( t,x\right) \right) ^{2}dx\rightarrow \int_{\Omega }\left(
u_{1n}^{i}\left( t,x\right) -u_{2n}^{i}\left( t,x\right) \right) ^{2}dx\text{
as }k\rightarrow \infty .
\end{equation*}%
Define the sequence 
\begin{equation*}
w_{n}^{ki}\left( t,x\right) =\left\{ 
\begin{array}{c}
u_{1n}^{ki}\left( t,x\right) -u_{2n}^{ki}\left( t,x\right) \text{, if }x\in
\Omega \backslash \Omega _{nk}^{it}, \\ 
0\text{, if }x\in \Omega _{nk}^{it}.%
\end{array}%
\right.
\end{equation*}%
Then it is clear that $\left\Vert w_{n}^{ki}\left( t\right) \right\Vert
_{L^{2}\left( \Omega \right) }\rightarrow \left\Vert u_{1n}^{ki}\left(
t\right) -u_{2n}^{ki}\left( t\right) \right\Vert _{L^{2}\left( \Omega
\right) }$ as $k\rightarrow \infty $. We note also that $w_{n}^{ki}\left(
t\right) \rightarrow u_{1n}^{i}\left( t\right) -u_{2n}^{i}\left( t\right) $
weakly in $L^{2}\left( \Omega \right) $, as for any $\xi \in L^{2}\left(
\Omega \right) ,$ (\ref{ConvOmegank}) gives%
\begin{eqnarray*}
&&\int_{\Omega }w_{n}^{ki}\left( t,x\right) \xi \left( x\right) dx \\
&=&\int_{\Omega }\left( u_{1n}^{ki}\left( t,x\right) -u_{2n}^{ki}\left(
t,x\right) \right) \xi \left( x\right) dx-\int_{\Omega _{nk}^{t}}\left(
u_{1n}^{ki}\left( t,x\right) -u_{2n}^{ki}\left( t,x\right) \right) \xi
\left( x\right) dx\rightarrow 0.
\end{eqnarray*}%
Therefore, $w_{n}^{ki}\left( t\right) \rightarrow u_{1n}^{i}\left( t\right)
-u_{2n}^{i}\left( t\right) $ strongly in $L^{2}\left( \Omega \right) $ so
that $w_{n}^{ki}\left( t,x\right) \rightarrow u_{1n}^{i}\left( t,x\right)
-u_{2n}^{i}\left( t,x\right) $ for a.a. $x\in \Omega .$ Since $%
w_{n}^{ki}\left( t,x\right) \leq 0$, for a.a. $\Omega $, we obtain 
\begin{equation*}
0\leq u_{1n}\left( t\right) \leq u_{2n}\left( t\right) \text{ for all }t\in
\lbrack \tau ,T].
\end{equation*}

Arguing again as in the proof of Theorem \ref{ComparisonWeak} we obtain that
the sequences $u_{1n},u_{2n}$ converge (up to a subsequence) in the sense of
(\ref{Conv1})-(\ref{Conv7}) to the solutions $u_{1},u_{2}$ of problem (\ref%
{P1}) and (\ref{P2}), respectively. Also, it holds%
\begin{equation*}
0\leq u_{1}\left( t\right) \leq u_{2}\left( t\right) ,\text{ for all }t\in
\lbrack \tau ,T].
\end{equation*}
\end{proof}

\bigskip

As in the previous section we shall generalize this theorem to the case
where the constant $\alpha $ can be negative. We note that if $f_{j},h_{j}$
satisfy (\ref{Cond11}), then $\widetilde{f}_{j}\left( t,v\right) =e^{-\beta
t}f_{j}(t,e^{\beta t}v)+\beta v$, $\widetilde{h}_{j}=e^{-\beta t}h_{j}$ also
satisfy (\ref{Cond11}). Arguing as in Theorem \ref{ComparisonWeak2} we
obtain the following.

\begin{theorem}
\label{ComparisonWeak2Positive}Let $f_{j},h_{j}$ satisfy (\ref{Cond10})-(\ref%
{Cond11}). Assume that $f_{j},h_{j}$ satisfy (\ref{Cond1b})-(\ref{Cond2b})
and (\ref{Ineqfh}) for $u\in \mathbb{R}_{+}^{d}$ . Also, suppose that either 
$f_{1}$ or $f_{2}$ satisfies (\ref{Cond4}) for $u\in \mathbb{R}_{+}^{d}$ for
an arbitrary $R_{0}>0$. If $u_{\tau }^{1}\leq u_{\tau }^{2},$ there exist
two solutions $u_{1},u_{2}$ (of (\ref{P1}) and (\ref{P2}), respectively),
with $u_{1}\left( \tau \right) =u_{\tau }^{1}$, $u_{2}\left( \tau \right)
=u_{\tau }^{2},$ such that $0\leq u_{1}\left( t\right) \leq u_{2}\left(
t\right) $, for all $t\in \lbrack \tau ,T].$
\end{theorem}

\bigskip

But in this case we can consider another interesting situation, as the
values $p_{i}^{1}$ and $p_{i}^{2}$ are not necessarily equal.

Let $f_{1}$ satisfy (\ref{Cond1})-(\ref{Cond2}) for $u\in \mathbb{R}_{+}^{d}$%
, whereas $f_{2}$ satisfy (\ref{Cond1b})-(\ref{Cond2b}) for $u\in \mathbb{R}%
_{+}^{d},$ with constants $\alpha ^{j},C_{1}^{j},C_{2}^{j},\ j=1,2$. Then if 
$\alpha ^{2}<0,$ we make in (\ref{Equation}) the change of variable $%
v=e^{-\beta t}u$, where $\beta >-\alpha ^{2}$. We obtain problems (\ref{P1b}%
)-(\ref{P2b}).

If $v\left( t\right) $ is a weak solution of (\ref{P1b}), then $u\left(
t\right) =e^{\beta t}v\left( t\right) $ is a weak solution of (\ref{P1})
(and the same is true, of course, for (\ref{P2b}) and (\ref{P2})).

The function $\widetilde{f}_{1}\left( t,v\right) =e^{-\beta
t}f_{1}(t,e^{\beta t}v)+\beta v$ satisfies (\ref{Cond1})-(\ref{Cond2}) for $%
v\in \mathbb{R}_{+}^{d}$ with the same $p_{i}^{1}$ as $f_{1}$. Indeed, as $%
e^{\beta t\left( p_{i}^{1}-2\right) }\geq 1$, we get%
\begin{eqnarray}
\sum_{i=1}^{d}\left\vert \widetilde{f}_{1}^{i}(t,v)\right\vert ^{\frac{%
p_{i}^{1}}{p_{i}^{1}-1}} &\leq &K_{1}\left( \sum_{i=1}^{d}\left\vert
f_{1}^{i}(t,e^{\beta t}v)\right\vert ^{\frac{p_{i}^{1}}{p_{i}^{1}-1}%
}+\sum_{i=1}^{d}\left\vert v^{i}\right\vert ^{\frac{p_{i}^{1}}{p_{i}^{1}-1}%
}\right)  \label{C3} \\
&\leq &K_{2}\left( 1+\sum_{i=1}^{d}\left\vert v^{i}\right\vert
^{p_{i}^{1}}\right) ,  \notag
\end{eqnarray}%
\begin{equation}
\left( \widetilde{f}_{1}(t,v),v\right) \geq e^{-2\beta t}\left( \alpha
^{1}\sum_{i=1}^{d}e^{\beta tp_{i}}\left\vert v^{i}\right\vert
^{p_{i}^{1}}-C_{2}^{1}\right) +\beta \left\vert v\right\vert ^{2}\geq \alpha
^{1}\sum_{i=1}^{d}\left\vert v^{i}\right\vert ^{p_{i}^{1}}-\widetilde{C}%
_{2}^{1}.  \label{C4}
\end{equation}

Then we obtain the following.

\begin{theorem}
\label{ComparisonWeak3Positive}Let $f_{j},h_{j}$ satisfy (\ref{Cond10})-(\ref%
{Cond11}). Let $f_{1}$ satisfy (\ref{Cond1})-(\ref{Cond2}) for $u\in \mathbb{%
R}_{+}^{d}$, whereas $f_{2}$ satisfy (\ref{Cond1b})-(\ref{Cond2b}) for $u\in 
\mathbb{R}_{+}^{d},$ with constants $\alpha ^{j},C_{1}^{j},C_{2}^{j},\ j=1,2$%
. Assume that $f_{j},h_{j}$ satisfy (\ref{Ineqfh}). Also, suppose that
either $f_{1}$ or $f_{2}$ satisfies (\ref{Cond4}) for $u\in \mathbb{R}%
_{+}^{d}$ for an arbitrary $R_{0}>0$. If $u_{\tau }^{1}\leq u_{\tau }^{2},$
there exist two solutions $u_{1},u_{2}$ (of (\ref{P1}) and (\ref{P2}),
respectively), with $u_{1}\left( \tau \right) =u_{\tau }^{1}$, $u_{2}\left(
\tau \right) =u_{\tau }^{2},$ such that $0\leq u_{1}\left( t\right) \leq
u_{2}\left( t\right) $, for all $t\in \lbrack \tau ,T].$
\end{theorem}

\begin{proof}
We consider problems (\ref{P1b}) and (\ref{P2b}). In view of (\ref{C1})-(\ref%
{C2}) and (\ref{C3})-(\ref{C4}), $\widetilde{f}_{2}\left( t,v\right)
=e^{-\beta t}f_{j}(t,e^{\beta t}v)+\beta v$ satisfy (\ref{Cond1})-(\ref%
{Cond2}) for $v\in \mathbb{R}_{+}^{d}$ with $p_{i}^{2}=2$, and $\widetilde{f}%
_{1}$ satisfies (\ref{Cond1})-(\ref{Cond2}) for $v\in \mathbb{R}_{+}^{d}$
with the same $p_{i}^{1}$ as $f_{1}$. Also, defining $\widetilde{h}%
_{j}\left( t,x\right) =e^{-\beta t}h_{j}(t,x)$ it is clear that (\ref{Ineqfh}%
) and (\ref{Cond11}) hold. Finally, if, for example, $f_{1}$ satisfies (\ref%
{Cond4}) for $u\in \mathbb{R}_{+}^{d}$ for any $R_{0}>0$, then it is obvious
that for $\widetilde{f}_{1}$ is true as well.

Hence, by Theorem \ref{ComparisonPositiveGeneral} there exist two solutions $%
v_{1},v_{2}$ (of (\ref{P1b}) and (\ref{P2b}), respectively), with $%
v_{j}\left( \tau \right) =e^{-\beta \tau }u_{\tau }^{j}$, such that $0\leq
v_{1}\left( t\right) \leq v_{2}\left( t\right) $, for all $t\in \lbrack \tau
,T].$ Thus%
\begin{equation*}
0\leq u_{1}\left( t\right) =e^{\beta t}v_{1}\left( t\right) \leq e^{\beta
t}v_{2}\left( t\right) =u_{2}\left( t\right) \text{, for }t\in \lbrack \tau
,T],
\end{equation*}%
and $u_{1},u_{2}$ are solutions (of (\ref{P1}) and (\ref{P2}), respectively,
such that $u_{j}\left( \tau \right) =u_{\tau }^{j}$.
\end{proof}

\begin{remark}
If $f_{j}$ satisfy (\ref{Cond3}), then the solutions $u_{1},u_{2}$ given in
Theorems \ref{ComparisonWeak2Positive}, \ref{ComparisonWeak3Positive} are
unique for the corresponding initial data.
\end{remark}

\begin{remark}
\label{Neumann}All the result proved so far are true if insead of Dirichlet
boundary conditions\ we consider Neumann boundary conditions%
\begin{equation*}
\frac{\partial u}{\partial \nu }=0\text{ in }\partial \Omega ,
\end{equation*}%
where $\nu $ is the unit outward normal. In such a case the space $V$ will
be $\left( H^{1}\left( \Omega \right) \right) ^{d}$. The proofs remain the
same.
\end{remark}

\section{Applications}

We shall apply now the previous results to some model of physical and
biological interest.

\subsection{The Lotka-Volterra system}

We also study the Lotka-Volterra system with diffusion%
\begin{equation}
\left\{ 
\begin{array}{cc}
\dfrac{\partial u^{1}}{\partial t}= & D_{1}\Delta u^{1}+u^{1}\left(
a_{1}\left( t\right) -u^{1}-a_{12}\left( t\right) u^{2}-a_{13}\left(
t\right) u^{3}\right) , \\ 
\dfrac{\partial u^{2}}{\partial t}= & D_{2}\Delta u^{2}+u^{2}\left(
a_{2}\left( t\right) -u^{2}-a_{21}\left( t\right) u^{1}-a_{23}\left(
t\right) u^{3}\right) , \\ 
\dfrac{\partial u^{3}}{\partial t}= & D_{3}\Delta u^{3}+u^{3}\left(
a_{3}\left( t\right) -u^{3}-a_{31}\left( t\right) u^{1}-a_{32}\left(
t\right) u^{2}\right) ,%
\end{array}%
\right.  \label{LV}
\end{equation}%
with either Dirichlet or Neumann boundary conditions, where $%
u^{i}=u^{i}\left( x,t\right) \geq 0$ and the functions $a_{i}\left( t\right)
,a_{ij}\left( t\right) $ are positive and continuous. Also, $D_{i}$ are
positive constants and $\Omega \subset \mathbb{R}^{3}$. The initial data $%
u_{\tau }$ belongs to $\left( L^{2}\left( \Omega \right) \right) ^{3}.$

In this case the functions $f_{1},h_{1}$ are given by%
\begin{equation*}
f_{1}\left( t,u\right) =\left( 
\begin{array}{c}
-u^{1}\left( a_{1}\left( t\right) -u^{1}-a_{12}\left( t\right)
u^{2}-a_{13}\left( t\right) u^{3}\right) \\ 
-u^{2}\left( a_{2}\left( t\right) -u^{2}-a_{21}\left( t\right)
u^{1}-a_{23}\left( t\right) u^{3}\right) \\ 
-u^{3}\left( a_{3}\left( t\right) -u^{3}-a_{31}\left( t\right)
u^{1}-a_{32}\left( t\right) u^{2}\right)%
\end{array}%
\right) ,\ h_{1}\left( t\right) \equiv 0.
\end{equation*}

Uniqueness of the Cauchy problem for this system has been proved only if we
consider solutions confined in an invariant region (for example, in a
parallelepiped $\mathcal{D}=\{\left( u^{1},u^{2},u^{3}\right) :0\leq
u^{i}\leq k^{i}\}$ when the parameters do not depend on $t$) (see \cite%
{Marion} and \cite{Smoller}). However, in the general case for initial data
just in $\left( L^{2}\left( \Omega \right) \right) ^{3}$ it is an open
problem so far.

System (\ref{LV}) satisfies conditions (\ref{Cond1})-(\ref{Cond2}) with $%
p_{1}=p_{2}=p_{3}=3$ for$\,u\in \mathbb{R}_{+}^{3}$ \cite[p.263]{KapVal09}.
Also, it is clear that (\ref{Cond10})-(\ref{Cond11}) hold.

We shall compare with the following system%
\begin{equation}
\left\{ 
\begin{array}{cc}
\dfrac{\partial u^{1}}{\partial t}= & D_{1}\Delta u^{1}+u^{1}\left(
a_{1}\left( t\right) -u^{1}\right) , \\ 
\dfrac{\partial u^{2}}{\partial t}= & D_{2}\Delta u^{2}+u^{2}\left(
a_{2}\left( t\right) -u^{2}\right) , \\ 
\dfrac{\partial u^{3}}{\partial t}= & D_{3}\Delta u^{3}+u^{3}\left(
a_{3}\left( t\right) -u^{3}\right) ,%
\end{array}%
\right.  \label{LVLogistic}
\end{equation}%
which is a system of three uncoupled logitic equations. The functions $%
f_{2},h_{2}$ are given by%
\begin{equation*}
f_{2}\left( t,u\right) =\left( 
\begin{array}{c}
-u^{1}\left( a_{1}\left( t\right) -u^{1}\right) \\ 
-u^{2}\left( a_{2}\left( t\right) -u^{2}\right) \\ 
-u^{3}\left( a_{3}\left( t\right) -u^{3}\right)%
\end{array}%
\right) ,\ h_{2}\left( t\right) \equiv 0.
\end{equation*}

It is easy to see that system (\ref{LVLogistic}) satisfies conditions (\ref%
{Cond1})-(\ref{Cond2}) with $p_{1}=p_{2}=p_{3}=3$ for$\,u\in \mathbb{R}^{3}$%
. Also, it is clear that (\ref{Cond10})-(\ref{Cond11}) hold, and that
condition (\ref{Cond4}) is trivially satisifed. Moreover,%
\begin{equation}
f_{1}^{i}\left( t,u\right) \geq f_{2}^{i}\left( t,u\right) ,  \notag
\end{equation}%
for all $t$, $u\in \mathbb{R}_{+}^{3}$ and $i$, so that (\ref{Ineqfh}) holds
for $u\in \mathbb{R}_{+}^{3}.$

Also, we have%
\begin{eqnarray*}
\left( f_{2u}\left( t,u\right) w,w\right) &\geq &-a_{1}\left( t\right)
w_{1}^{2}-a_{2}\left( t\right) w_{2}^{2}-a_{3}\left( t\right) w_{3}^{2} \\
&&+2\left( u^{1}w_{1}^{2}+u^{2}w_{2}^{2}+u^{3}w_{3}^{2}\right) \geq
-C\left\vert w\right\vert ^{2},
\end{eqnarray*}%
for all $u\in \mathbb{R}_{+}^{3}$, $t\in \lbrack \tau ,T]$ and $w\in \mathbb{%
R}^{3}$, where $C>0.$

Thus, by Theorem \ref{ComparisonPositiveGeneral}, Lemma \ref{Unique} and
Remark \ref{Neumann} we obtain the following theorem.

\begin{theorem}
If $0\leq u_{\tau }^{1}\leq u_{\tau }^{2},$ there exist a solution $u_{1}$
of (\ref{LV}) with $u_{1}\left( \tau \right) =u_{\tau }^{1}$ such that $%
0\leq u_{1}\left( t\right) \leq u_{2}\left( t\right) $, for all $t\in
\lbrack \tau ,T]$, where $u_{2}$ is the unique solution with $u_{2}\left(
\tau \right) =u_{\tau }^{2}$ of (\ref{LVLogistic}) in the class of solutions
satisfying $u_{2}\left( t\right) \geq 0$ for all $t.$
\end{theorem}

This theorem says that there exists at least one solution of the
Lotka-Volterra system which is dominated by the unique non-negative solution
of the uncoupled logistic system (\ref{LVLogistic}).

\bigskip

Further, we shall compare with the uncoupled\ linear system%
\begin{equation}
\left\{ 
\begin{array}{cc}
\dfrac{\partial u^{1}}{\partial t}= & D_{1}\Delta u^{1}+u^{1}a_{1}\left(
t\right) , \\ 
\dfrac{\partial u^{2}}{\partial t}= & D_{2}\Delta u^{2}+u^{2}a_{2}\left(
t\right) , \\ 
\dfrac{\partial u^{3}}{\partial t}= & D_{3}\Delta u^{3}+u^{3}a_{3}\left(
t\right) .%
\end{array}%
\right.  \label{LVLinear}
\end{equation}%
Hence, 
\begin{equation*}
f_{2}\left( t,u\right) =\left( 
\begin{array}{c}
-u^{1}a_{1}\left( t\right) \\ 
-u^{2}a_{2}\left( t\right) \\ 
-u^{3}a_{3}\left( t\right)%
\end{array}%
\right) ,\ h_{2}\left( t\right) \equiv 0.
\end{equation*}

Obviously system (\ref{LVLinear}) satisfies conditions (\ref{Cond1b})-(\ref%
{Cond2b}) for$\,u\in \mathbb{R}^{3}$. Also, it is clear that (\ref{Cond10})-(%
\ref{Cond11}) and condition (\ref{Cond4}) are trivially satisifed. Moreover,%
\begin{equation}
f_{1}^{i}\left( t,u\right) \geq f_{2}^{i}\left( t,u\right) ,  \notag
\end{equation}%
for all $t$, $u\in \mathbb{R}_{+}^{3}$ and $i$, so that (\ref{Ineqfh}) holds
for $u\in \mathbb{R}_{+}^{3}.$ Also, we have%
\begin{equation*}
\left( f_{2u}\left( t,u\right) w,w\right) \geq -a_{1}\left( t\right)
w_{1}^{2}-a_{2}\left( t\right) w_{2}^{2}-a_{3}\left( t\right) w_{3}^{2}\geq
-C\left\vert w\right\vert ^{2},
\end{equation*}%
for all $u\in \mathbb{R}^{3}$, $t\in \lbrack \tau ,T]$ and $w\in \mathbb{R}%
^{3}$, where $C>0.$

Thus, by Theorem \ref{ComparisonWeak3Positive} and Remark \ref{Neumann} we
have the following.

\begin{theorem}
\label{ComparisonLVLinear}If $0\leq u_{\tau }^{1}\leq u_{\tau }^{2},$ there
exist a solution $u_{1}$ of (\ref{LV}) with $u_{1}\left( \tau \right)
=u_{\tau }^{1}$ such that $0\leq u_{1}\left( t\right) \leq u_{2}\left(
t\right) $, for all $t\in \lbrack \tau ,T]$, where $u_{2}$ is the unique
solution of (\ref{LVLinear}) with $u_{2}\left( \tau \right) =u_{\tau }^{2}.$
\end{theorem}

\bigskip

Let us consider now the autonomous case, that is, $a_{i}\left( t\right)
\equiv a_{i}>0$ and $\tau =0$, with Diriclet boundary conditions.

By the changes of variable $v^{i}\left( t\right) =e^{-a_{i}t}u^{i}\left(
t\right) $ system (\ref{LVLinear}) becomes 
\begin{equation}
\left\{ 
\begin{array}{c}
\dfrac{\partial v^{i}}{\partial t}=D_{i}\Delta v^{i},\ i=1,2,3, \\ 
v^{i}|_{x\in \partial \Omega }=0,%
\end{array}%
\right.  \label{LVLinear2}
\end{equation}%
with initial data $v_{0}=u_{0}.$

The operator $A=-D_{i}\Delta $ in the space $L^{2}\left( \Omega \right) $
with domain $D\left( A\right) =H^{2}\left( \Omega \right) \cap
H_{0}^{1}\left( \Omega \right) $ is sectorial \cite{Henry}. Moreover, since
the eigenvalues of $A$ are $0<D_{i}\lambda _{1}\leq D_{i}\lambda _{2}\leq
... $, we have that the minimum eigenvalue is strictly greater than $0$.
Denote by $e^{-At}$ the analytic semigroup generated by the operator $A$.
Then $v\left( t\right) =e^{-At}v_{0}$ is the unique solution of (\ref%
{LVLinear2})\ with $v\left( 0\right) =v_{0}$.

It is well known \cite{Henry} that the operator $A$ generate a scale of
interpolation spaces $X^{\alpha }=D\left( A^{\alpha }\right) $ with the norm 
$\left\Vert v\right\Vert _{\alpha }=\left\Vert A^{\alpha }v\right\Vert
_{L^{2}\left( \Omega \right) }$, $\alpha \geq 0$, where $X^{\alpha }\subset
H^{2\alpha }\left( \Omega \right) $ with continuous embedding. Take $%
0<\delta <D_{1}\lambda _{1}$. Then by Theorem 1.4.3 in \cite{Henry} we obtain%
\begin{equation*}
\left\Vert A^{\alpha }e^{-At}v_{0}^{i}\right\Vert _{L^{2}\left( \Omega
\right) }\leq C_{\alpha }t^{-\alpha }e^{-\delta t}\left\Vert
v_{0}^{i}\right\Vert _{L^{2}\left( \Omega \right) }\text{ for }t>0.
\end{equation*}%
Since this is true for every $\delta <D_{i}\lambda _{1}$ we obtain that%
\begin{eqnarray*}
\left\Vert A^{\alpha }e^{-At}v_{0}^{i}\right\Vert _{L^{2}\left( \Omega
\right) } &\leq &C_{\alpha }t^{-\alpha }e^{-D_{i}\lambda _{1}t}\left\Vert
v_{0}^{i}\right\Vert _{L^{2}\left( \Omega \right) }, \\
\left\Vert e^{-At}v_{0}^{i}\right\Vert _{H^{2\alpha }\left( \Omega \right) }
&\leq &\widetilde{C}_{\alpha }t^{-\alpha }e^{-D_{i}\lambda _{1}t}\left\Vert
v_{0}^{i}\right\Vert _{L^{2}\left( \Omega \right) }\text{ for }t>0.
\end{eqnarray*}%
We note that $H^{2\alpha }\left( \Omega \right) \subset L^{\infty }\left(
\Omega \right) $ with continuous embedding if $\alpha >\frac{3}{4}$. Since
the constants $C_{\alpha }$ are bounded for $\alpha $ in compact sets, we
obtain the existence of $C$ such that%
\begin{equation*}
\left\Vert e^{-At}v_{0}^{i}\right\Vert _{L^{\infty }\left( \Omega \right)
}\leq Ct^{-\frac{3}{4}}e^{-D\lambda _{1}t}\left\Vert v_{0}^{i}\right\Vert
_{L^{2}\left( \Omega \right) }\text{ for }t>0,
\end{equation*}%
where $D=\min \{D_{1},D_{2},D_{3}\}$. Then the unique solution of (\ref%
{LVLinear}) with $u\left( 0\right) =u_{0}$ satisfies%
\begin{equation}
\left\Vert u\left( t\right) \right\Vert _{\left( L^{\infty }\left( \Omega
\right) \right) ^{3}}\leq Ct^{-\frac{3}{4}}e^{\left( a-D\lambda _{1}\right)
t}\left\Vert u_{0}\right\Vert \text{ for }t>0,  \label{EstInf}
\end{equation}%
where $a=\max \{a_{1},a_{2},a_{3}\}$.

Joining (\ref{EstInf}) and Theorem \ref{ComparisonLVLinear} we obtain the
following result.

\begin{theorem}
There exist at least one solution $u\left( t\right) $ of the autonomous
system (\ref{LV}) with $u\left( 0\right) =u_{0}$ and Dirichlet boundary
conditions such that%
\begin{equation*}
\left\Vert u\left( t\right) \right\Vert _{\left( L^{\infty }\left( \Omega
\right) \right) ^{3}}\leq Ct^{-\frac{3}{4}}e^{\left( a-D\lambda _{1}\right)
t}\left\Vert u_{0}\right\Vert \text{ for }t\in (0,T],
\end{equation*}%
where $C>0$ and $D=\min \{D_{1},D_{2},D_{3}\},$ $a=\max
\{a_{1},a_{2},a_{3}\}.$
\end{theorem}

We shall obtain also a weak maximum principle for the autonomous
Lotka-Volterra system with Dirichlet boundary conditions.

By the maximum principle for \ the heat equation it is well known (see \cite%
{Brezis}) that for any $t\geq 0$ the unique solution of equation (\ref%
{LVLinear2}) satisfies%
\begin{equation*}
0\leq v^{i}\left( t,x\right) \leq \sup_{\Omega }\ v_{0}^{i}\text{ for a.a. }%
x\in \Omega \text{.}
\end{equation*}%
Hence, the unique solution of equation (\ref{LVLinear}) satisfies%
\begin{equation}
0\leq u^{i}\left( t,x\right) \leq e^{a_{i}t}\sup_{\Omega }\ v_{0}^{i}\text{
for a.a. }x\in \Omega \text{.}  \label{Max}
\end{equation}

By (\ref{Max}) and Theorem \ref{ComparisonLVLinear} we obtain the following
weak maximum principle.

\begin{theorem}
There exist at least one solution $u\left( t\right) $ of the autonomous
system (\ref{LV}) with $u\left( 0\right) =u_{0}$ and Dirichlet boundary
conditions such that for any $t\in \lbrack 0,T],$%
\begin{equation*}
0\leq u^{i}\left( t,x\right) \leq e^{a_{i}t}\sup_{\Omega }\ u_{0}^{i}\text{
for a.a. }x\in \Omega \text{, }i=1,2,3\text{.}
\end{equation*}

In particular, if $u_{0}\in \left( L^{\infty }\left( \Omega \right) \right)
^{3}$, then%
\begin{equation*}
\left\Vert u\left( t\right) \right\Vert _{\left( L^{\infty }\left( \Omega
\right) \right) ^{3}}\leq e^{at}\left\Vert u_{0}\right\Vert _{\left(
L^{\infty }\left( \Omega \right) \right) ^{3}}\text{ for }t\in \lbrack 0,T],
\end{equation*}%
where $a=\max \{a_{1},a_{2},a_{3}\}.$
\end{theorem}

\subsection{A model of fractional-order chemical autocatalysis with decay}

Consider the following scalar problem 
\begin{equation}
\left\{ 
\begin{array}{c}
\dfrac{\partial u}{\partial t}=\dfrac{\partial ^{2}u}{\partial x^{2}}+\left(
1-u\right) u^{m}-ku^{r}, \\ 
\dfrac{\partial u}{\partial x}\left( 0,t\right) =\dfrac{\partial u}{\partial
x}\left( a,t\right) =0, \\ 
u|_{t=0}=u_{0}\left( x\right) ,%
\end{array}%
\right.  \label{autocatalyst}
\end{equation}%
where $u\geq 0$, $N=1$, $d=1,$ $\Omega =\left( 0,L\right) $, and $k>0$, $%
0<m,r<1$. The initial data $u_{0}$ belongs to $L^{2}\left( 0,a\right) $.
This equation models an isothermal chemical autocatalysis (see \cite{McCabe}%
). In \cite{McCabe} the authors study the travelling waves of the equation
in the case where $\Omega =\left( 0,+\infty \right) $ with Neumann boundary
conditions at $x=0$. The variable $u$ is non-negative, since it represents a
chemical concentration.

The funtions $f,h$ are given by $f\left( u\right) =\left( u-1\right)
u^{m}+ku^{r}$, $h\equiv 0$. Clearly, conditions (\ref{Cond1})-(\ref{Cond2})
hold (with$\ p=m+2$) for $u\geq 0$. In this case (\ref{Cond10})-(\ref{Cond11}%
) and (\ref{Cond4}) are trivially satisifed.

We take $f_{1}=f_{2}=f$, and applying Theorem \ref{ComparisonPositiveGeneral}
and Remark \ref{Neumann} obtain the following.

\begin{theorem}
If $0\leq u_{\tau }^{1}\leq u_{\tau }^{2},$ there exist solutions $%
u_{1},u_{2}$ of (\ref{autocatalyst}) with $u_{1}\left( 0\right) =u_{0}^{1},\
u_{2}\left( 0\right) =u_{0}^{2}$ such that $0\leq u_{1}\left( t\right) \leq
u_{2}\left( t\right) $, for all $t\in \lbrack 0,T]$.
\end{theorem}

\bigskip

\subsection{A generalized logistic equation}

Consider the following scalar problem%
\begin{equation}
\left\{ 
\begin{array}{c}
\dfrac{\partial u}{\partial t}=\dfrac{\partial ^{2}u}{\partial x^{2}}+\left(
1-u^{q}\right) u^{r}, \\ 
\dfrac{\partial u}{\partial x}\left( 0,t\right) =\dfrac{\partial u}{\partial
x}\left( a,t\right) =0, \\ 
u|_{t=0}=u_{0}\left( x\right) ,%
\end{array}%
\right.  \label{Logistic}
\end{equation}%
where $u\geq 0$, $N=1$, $\Omega =\left( 0,L\right) $, $r,q>0$ and $r+q\geq 1$%
. The initial data $u_{0}$ belongs to $L^{2}\left( 0,a\right) $.

This kind of nonlinearities for the logistic equation (instead of the
classical $\left( 1-u\right) u$) has been considered in \cite[Chapter 11]%
{Murray}.

The funtions $f,h$ are given by $f\left( u\right) =\left( u^{q}-1\right)
u^{r}$, $h\equiv 0$. Clearly, conditions (\ref{Cond1})-(\ref{Cond2}) hold
(with $p=r+q+1$) for $u\geq 0$. In this case (\ref{Cond10})-(\ref{Cond11})
and (\ref{Cond4}) are trivially satisifed.

We take $f_{1}=f_{2}=f$, and applying Theorem \ref{ComparisonPositiveGeneral}
and Remark \ref{Neumann} obtain the following.

\begin{theorem}
If $0\leq u_{\tau }^{1}\leq u_{\tau }^{2},$ there exist two solution $%
u_{1},u_{2}$ of (\ref{Logistic}) with $u_{1}\left( 0\right) =u_{0}^{1},\
u_{2}\left( 0\right) =u_{0}^{2}$ such that $0\leq u_{1}\left( t\right) \leq
u_{2}\left( t\right) $, for all $t\in \lbrack 0,T]$.
\end{theorem}

\textbf{Acknowledgments.}

This work was been partially supported by the Ministerio de Ciencia e
Innovaci\'{o}n, Projects MTM2008-00088 and MTM2009-11820, by the Consejer%
\'{\i}a de Innovaci\'{o}n, Ciencia y Empresa (Junta de Andaluc\'{\i}a),
grant P07-FQM-02468, and by the Consejer\'{\i}a de Cultura y Educaci\'{o}n
(Comunidad Aut\'{o}noma de Murcia), grant 08667/PI/08.

\end{document}